\documentclass [11pt]{amsart}
\usepackage{amsmath}
\usepackage{amssymb}
\usepackage{comment}

\newtheorem{thm}{Theorem}[section]
\newtheorem{cor}[thm]{Corollary}
\newtheorem{lem}[thm]{Lemma}

\theoremstyle{definition}
\newtheorem{defn}[thm]{Definition}
\theoremstyle{remark}
\newtheorem{rem}[thm]{Remark}
\numberwithin{equation}{section}

\newcommand{\beas}{\begin{eqnarray*}}
\newcommand{\eeas}{\end{eqnarray*}}
\newcommand{\bes} {\begin{equation*}}
\newcommand{\ees} {\end{equation*}}
\newcommand{\be} {\begin{equation}}
\newcommand{\ee} {\end{equation}}
\newcommand{\bea} {\begin{eqnarray}}
\newcommand{\eea} {\end{eqnarray}}
\newcommand{\ra} {\rightarrow}
\newcommand{\txt} {\textmd}
\newcommand{\ds} {\displaystyle}
\newcommand{\R}{\mathbb R}
\newcommand{\C}{\mathbb C}
\newcommand{\T}{\mathbb T}
\newcommand{\N}{\mathbb N}
\newcommand{\Z}{\mathbb Z}

\begin{document}
\title[\tiny{Around Uncertainty Principles of Ingham-type on $\R^n$, $\T^n$ and 2-Step Nilpotent Lie Groups}]{Around Uncertainty Principles of Ingham-type on $\R^n$, $\T^n$ and Two Step Nilpotent Lie Groups}

\author{\tiny{Mithun Bhowmik, Swagato K. Ray and Suparna Sen}}

\address{Stat-Math Unit, Indian Statistical Institute, 203 B. T. Road, Kolkata - 700108,  India.}

\email{mithunbhowmik123@gmail.com, swagato@isical.ac.in, suparna29@gmail.com}

\thanks{This work was supported by Indian Statistical Institute, India (Research fellowship to Mithun Bhowmik); and Department of Science and Technology, India (INSPIRE Faculty Award to Suparna Sen).}


\begin{abstract}
Classical results due to Ingham and Paley-Wiener characterize the existence of nonzero functions supported on certain subsets of the real line in terms of the pointwise decay of the Fourier transforms. We view these results as uncertainty principles for Fourier transforms. We prove certain analogues of these uncertainty principles on the $n$-dimensional Euclidean space, the $n$-dimensional torus and connected, simply connected two step nilpotent Lie groups. We also use these results to show a unique continuation property of solutions to the initial value problem for time-dependent Schr\"odinger equations on the Euclidean space and a class of connected, simply connected two step nilpotent Lie groups.
\end{abstract}



\maketitle

\section{Introduction}
Uncertainty principles in harmonic analysis are results which are based on the general principle that a nonzero function and its Fourier transform both cannot be small simultaneously. Several versions of uncertainty principles have been proved over the years on Euclidean spaces and other noncommutative groups (see \cite{FS, T}). In this paper we shall be interested in some classical uncertainty principles due to Ingham and Paley-Wiener which characterize the existence of a nonzero function supported on a certain subset of the real line in terms of the pointwise decay of its Fourier transform. For example, if we consider a compactly supported function $f \in L^1(\R)$ such that its Fourier transform satisfies the estimate 
\be \label{est} 
|\widehat f(y)| \leq Ce^{-a|y|}, \:\:\:\: \txt{ for some } a>0,
\ee 
then $f$ extends as a holomorphic function to a strip in the complex plane containing the real line and hence $f$ is identically zero. 

Instead of the positive real number $a$ in (\ref{est}), in \cite{I} Ingham considered a nonnegative even function $\theta(y)$ on $\R$ decreasing to $0$ as $|y| \ra \infty$ and gave a characterization of the existence of a nonzero compactly supported function in terms of an integrability condition on $\theta$. 
\begin{thm} \label{ingham}
Let $\theta(y)$ be a nonnegative even function on $\R$ such that $\theta(y)$ decreases to zero when $y \rightarrow \infty$. There exists a nonzero continuous function $f$ on $\R,$ equal to zero outside any interval $(-l,l),$ having Fourier transform $\widehat f$ satisfying the estimate 
\bes |\widehat{f}(y)|=O(e^{-\theta(y)|y|}), \:\:\:\: \txt{ when } |y| \ra \infty,
\ees 
if and only if  
\be \label{condR} 
\int_1^{\infty}\frac{\theta(t)}{t}dt<\infty.
\ee
\end{thm}
\noindent In fact, we shall see that a stronger result holds true. Precisely, instead of considering a nonzero compactly supported function $f$ it is enough to consider a nonzero function $f$ which vanishes on any open interval to prove the necessity of condition (\ref{condR}) (see Theorem \ref{open}).  

On the other hand, if we consider a locally integrable function $\psi(y)$ instead of $a|y|$ in (\ref{est}), a result due to Paley and Wiener (Theorem II, \cite{PW1}; Theorem XII, P. 16, \cite{PW}) gives  a characterization of the existence of a nonzero $f \in L^2(\R)$ vanishing on a half-line whose Fourier transform satisfies an estimate analogous to (\ref{est}) in terms of an integrability condition on $\psi$. 

\begin{thm}\label{paleywiener}
Let $\psi$ be a nonnegative locally integrable even function on $\R$. There exists a nonzero $f \in L^2(\R)$ vanishing for $x \geq x_0$ for some $x_0 \in \R$ such that 
\bes
|\widehat{f}(y)| = O(e^{-\psi(y)}), \:\:\:\: \txt{ for almost every } y \txt{ as } |y| \ra \infty,
\ees 
if and only if 
\bes 
\int_{\R} \frac{\psi (t)}{1+t^2}dt<\infty.
\ees
\end{thm}
\noindent It is easy to see that the necessity of the condition (\ref{condR}) in Theorem \ref{ingham} follows readily from Theorem \ref{paleywiener}. However, as mentioned before, we are interested in a stronger version of Ingham's uncertainty principle which will not follow from Theorem \ref{paleywiener}. Similar results on the real line have also been studied in \cite{PW1, L2, Hi, PW, L1, K}. Recently we have established similar results for continuous compactly supported functions on the Euclidean motion group in \cite{BS}.

The main purpose of this paper is to prove certain analogues of Theorem \ref{ingham} and Theorem \ref{paleywiener} on the Euclidean space $\R^n$, the $n$-dimensional torus $\mathbb T^n$ and connected, simply connected two step nilpotent Lie group $G$. We shall prove a several variable analogue (Theorem \ref{open}) of the stronger version of Theorem \ref{ingham} where we consider a nonzero function vanishing on any open set in $\R^n$. We shall also prove a several variable analogue (Theorem \ref{halfspace}) of Theorem \ref{paleywiener} where we consider functions supported on a half-space in $\R^n$. Since there is no obvious analogue of compactly supported functions or functions supported on a half-space in $\T^n$, we shall prove an analogue (Theorem \ref{torus}) of Theorem \ref{open} on $\T^n$ where we consider a nonzero function vanishing on any open subset in $\T^n$. To prove analogues (Theorem \ref{openG} and Theorem \ref{halfspG}) of Theorem \ref{open} and Theorem \ref{halfspace} on a connected, simply connected two step nilpotent Lie group $G$, we shall consider functions vanishing on open sets and half-spaces respectively in the center of $G$.   

It has recently been observed in \cite{EKPV, KPV, C, BTD, LM, PS} that it is possible to relate some of the uncertainty principles to the study of unique continuation properties of solutions to the Schr\"odinger Equation. We shall prove such a result (Theorem \ref{schrRn}) relating Theorem \ref{halfspace} to the problem of unique continuation of solutions to the Schr\"odinger Equation on $\R^n$ corresponding to a large class of second order left invariant differential operators. Our last result (Theorem \ref{uniqG}) in this paper attempts to relate the uncertainty principles of both Ingham and Paley-Wiener to the unique continuation property of solutions to the initial value problem for the time-dependent Schr\"odinger Equation on a class of connected, simply connected two step nilpotent Lie groups. 

The paper is organized as follows: In the next section we prove analogues of a stronger version of Theorem \ref{ingham} and Theorem \ref{paleywiener} for functions on $\R^n$ and an analogue of Theorem \ref{open} on $\mathbb T^n$. In section 3, we first describe the required preliminaries on a connected, simply connected two step nilpotent Lie group $G$ and then prove analogues of Theorem \ref{open} and Theorem \ref{halfspace} on $G$. In the last section we prove the unique continuation property of solutions to the initial value problem for time-dependent Schr\"odinger equation for a large class of operators on $\R^n$ and subsequently for the sublaplacian on a certain class of connected, simply connected two step groups. 

We shall use the following notation in this paper: $C_c(X)$ denotes the set of compactly supported continuous functions on $X$, $C_c^{\infty}(X)$ denotes the set of compactly supported smooth functions on $X$, $\mathcal S(X)$ denotes the space of Schwartz class functions on $X$, $\txt{supp}~f$ denotes the support of the function $f$ and $C$ denotes a constant whose value may vary. For a finite set $A$ we shall use the symbol $\#A$ to denote the number of elements in $A$. $B(0,r)$ denotes the ball of radius $r$ centered at $0$ in $\R^n$. For $x,y \in \R^n$, we shall use $|x|$ to denote the norm of the vector $x$ and $x \cdot y$ to denote the Euclidean inner product of the vectors $x$ and $y$. We shall interpret a radial function on $\R^n$ as an even function on $\R$ or a function on $[0,\infty)$ whenever required.

\section{Uncertainty Principles of Ingham and Paley-Wiener on $\R^n$ and $\T^n$}
In this section we prove analogues of the results of Ingham and Paley-Wiener on $\R^n$ and $\T^n$. 

\subsection{Uncertainty Principles of Ingham and Paley-Wiener on $\R^n$}
In this subsection our aim is to prove several variable analogues of Theorem \ref{ingham} and Theorem \ref{paleywiener}. Our method of proof uses the several variable analogue of the classical Denjoy-Carleman theorem regarding quasi-analytic class of functions proved by Bochner and Taylor \cite{BT} and the notion of Radon transform.  

First we briefly recall some standard facts regarding Radon transform. For $\omega \in S^{n-1}$ and $t \in \R$ let $H_{\omega,t}=\{x\in\R^n:x \cdot \omega=t\}$ denote the hyperplane on $\R^n$ with normal $\omega$ and distance $|t|$ from the origin. It is clear that $H_{\omega,t}=H_{-\omega,-t}$. For $f \in C_c(\R^n)$ the Radon transform $Rf$ of the function $f$ is defined by
\bes
Rf(\omega,t)=\int_{H_{\omega ,t}}f(x)dm(x),
\ees 
where $dm(x)$ is the $n-1$ dimensional Lebesgue measure of $H_{\omega,t}$. For $f \in L^1(\R^n)$ we define the Fourier transform $\widehat{f}$ of $f$ by 
\bes
\widehat{f} (y) = \int_{\R^n} f(x) e^{-2\pi ix \cdot y} dx, \:\:\:\: y \in \R^n.
\ees 
The one dimensional Fourier transform of $Rf$ and the Fourier transform of $f$ are closely connected by the slice projection theorem (P. 4, \cite{He}):
\be \label{sliceproj}
\widehat{f}(\lambda\omega)= \mathcal F_1(Rf(\omega, \cdot))(\lambda),
\ee
where $\mathcal F_1(Rf(\omega, \cdot))$ denotes the one dimensional Fourier transform of the function $t\mapsto Rf(\omega ,t)$. Clearly, if $f$ is a radial function on $\R^n$, then $Rf(\omega,t)$ is independent of $\omega$ and we can consider $Rf$ as an even function on $\R$. 
Let $C_c^{\infty}(\R^n)_0$ denote the set of compactly supported, smooth, radial functions on $\R^n$ and $C_c^{\infty}(\R)_e$ denote the set of compactly supported, smooth, even functions on $\R$. By Theorem 2.10 of \cite{He} it is known that
\be \label{radonmapping}
R:C_c^{\infty}(\R^n)_0\longrightarrow C_c^{\infty}(\R)_e
\ee
is a bijection with the property that if $g\in C_c^{\infty}(\R)_e$ with $\txt{supp}~g\subset [-r,r]$ then there exists a unique $f\in C_c^{\infty}(\R^n)_0$ with $\txt{supp}~f\subset B(0,r)$ and $Rf=g$.

Now, we need the several variable analogue of the classical Denjoy-Carleman theorem regarding quasi-analytic class of functions. To state the result we need some more notation. For a $C^{\infty}$ function $f$ on $\R^n$ we define $D_0f(x)=|f(x)|$ and 
\bes 
D_kf(x)=\left\{\sum_{|\alpha|=k} \left|\frac{{\partial}^kf(x)} {\partial x_1^{\alpha_1} \cdots \partial x_n^{\alpha_n}} \right|^2\right\}^{1/2},
\ees 
where $k \in \N$, $\alpha=(\alpha_1,\alpha_2,\cdots,\alpha_n)\in \{\N \cup \{0\}\}^n$ and $|\alpha|=\alpha_1+\alpha_2+ \cdots +\alpha_n$.

\begin{thm}[Theorem 1, \cite{BT}] \label{bochnertaylor}
Let $f$ be a smooth function defined on a connected domain $\Omega \subset \R^n$ and $x_0$ be an interior point of $\Omega$. Then the conditions
\begin{enumerate}
\item[i)] $D_kf(x)\leq m_k$ for all $x\in \Omega$, $k \in \N \cup \{0\}$,  
\item[ii)] $D_kf(x_0)=0$ for all $k \in \N \cup \{0\}$,
\item[iii)] $\ds{\sum_{k \in \N}{m_k^{-\frac{1}{k}}}=\infty}$ 
\end{enumerate}
imply that $f$ is zero throughout $\Omega$.
\end{thm}

Now we are in a position to present an analogue of the stronger version of Theorem \ref{ingham}.
\begin{thm}\label{open}
Let $\theta:\R^n\rightarrow [0,\infty)$ be a decreasing radial function with $\lim_{|y|\to\infty}\theta(y)=0$ and
\bes
I=\int_{|y|\geq 1}\frac{\theta (y)}{|y|^n}dy.
\ees
\begin{enumerate}
\item[(a)] Let $f\in L^1(\R^n)$ be a function satisfying the estimate
\be \label{ingcond2}
|\widehat{f}(y)|\leq Ce^{-\theta (y)|y|},
\ee
for all $y \in \R^n$. If $f$ vanishes on a nonempty open set and $I=\infty$ then $f(x)=0$ for almost every $x\in\R^n$.
\item[(b)] If $I$ is finite then there exists a nontrivial $f\in C_c^{\infty}(\R^n)$ satisfying (\ref{ingcond2}).
\end{enumerate}
\end{thm}

\begin{proof}
We shall first prove (b). We shall show that if $I$ is finite then given any $l>0$ there exists a nontrivial $f\in C_c^{\infty}(\R^n)$ which is supported in $B(0,l)$ and satisfies (\ref{ingcond2}). First of all, we can obtain a nontrivial function $g_1\in C_c(\R)$ with $\txt{supp}~g_1 \subset [-l/4,l/4]$ such that
\bes
|\widehat{g_1}(y)|\leq Ce^{-\theta(y)|y|},
\ees
for all $y \in \R$ by Theorem \ref{ingham}. Here we have used radiality of the function $\theta$ to interpret it as an even function on $\R$.
Let $\phi\in C_c^\infty(\R)$ be a nontrivial function with $\txt{supp}~\phi \subset [-l/4,l/4]$. We consider the function $g=g_1*\phi\in C_c^\infty(\R)$ and note that $\txt{supp}~g \subset [-l/2,l/2]$ and 
\be \label{decayg}
|\widehat g(y)| \leq C e^{-\theta(y)|y|},
\ee 
for all $y \in \R$. 
Since the condition (\ref{decayg}) is true for translates of $g$, by translating $g$ if necessary, we can get a nontrivial even function $(g(x)+g(-x))/2$ supported on $(-l,l)$ whose Fourier transform also satisfies the condition (\ref{decayg}). 
So we can assume the function $g$ to be nontrivial and even with $\txt{supp}~g \subset (-l,l)$. 
From (\ref{radonmapping}) and (\ref{sliceproj}) it now follows that there exists a nontrivial $f\in C_c^\infty(\R^n)_0$ such that $Rf=g$ with $\txt{supp}~f \subset B(0,l)$ and for all $y \in \R^n$
\bes
|\widehat{f}(y)|=|\mathcal F_1  Rf(|y|)|=|\widehat g(|y|)| \leq C e^{-\theta(y)|y|}.
\ees  
This, in particular, proves (b).
 
We now prove (a) under the following restriction on the function $\theta$, 
\be \label{cond} 
\theta(y) \geq 2|y|^{-\frac{1}{2}}, \:\:\:\: \txt{ for all } |y|\geq 1. 
\ee
Using this restriction and (\ref{ingcond2}) it follows that $\widehat{f}$ satisfies the estimate
\be \label{exponential} 
|\widehat{f}(y)| \leq C e^{-2\sqrt{|y|}}, \:\:\:\: \txt{ for all } |y|\geq 1. 
\ee
Using this exponential decay it follows that $\widehat{f}\in L^1(\R^n)$ and by the Fourier inversion $f \in C^{\infty}(\R^n)$. Hence
\bea \label{partial}
\left| \frac{\partial^k f(x)}{\partial x_1^{\alpha_1} \ldots \partial x_n^{\alpha_n} }\right| &=& \left|\frac{\partial^k}{\partial x_1^{\alpha_1}\ldots\partial x_n^{\alpha_n} }\left(\int_{\R^n}\widehat{f}(\xi)e^{2\pi i\xi\cdot y}d\xi\right)\right| \nonumber \\ &\leq& (2\pi)^k\int_{\R^n}|\xi_1|^{\alpha_1}\ldots|\xi_n|^{\alpha_n}|\widehat{f}(\xi)|d\xi,
\eea 
for $|\alpha|=k$. Differentiation under the integral sign is being justified by (\ref{exponential}). We now want to apply Theorem \ref{bochnertaylor} to the function $f$. This requires estimating the supremum of the functions $D_kf$ which we do now.
If $k>n-1$ then by (\ref{partial}) we get
\beas
D_kf(x)
&\leq& (2\pi)^k\left[\sum_{|\alpha|=k}\left(\int_{\R^n}{|\xi_1|^{\alpha_1} \cdots |\xi_n|^{\alpha_n} |\widehat{f}(\xi)|d\xi}\right)^2\right]^{1/2}\\
&\leq& (2\pi)^k\left[\sum_{|\alpha|=k}\left(\int_{\R^n}{|\xi|^k|\widehat f(\xi)|d\xi}\right)^2\right]^{1/2}\\
&\leq& C(2\pi)^k {{k+n-1}\choose k}^{1/2} \int_{\R^n} |\xi|^k e^{-\theta(\xi)|\xi|}d\xi\\
&\leq& C(2\pi)^k k^{n/2}\left(\int_0^{k^4}{r^{k+n-1}e^{-\theta(r)r}dr} + \int_{k^4}^{\infty}{r^{k+n-1}e^{-\theta(r)r}dr}\right).
\eeas
Here we have interpreted the radial function $\theta$ as a function on $\R$ and have used the elementary estimate for $k>n-1$,
\be \label{count}
\#\left\{(\alpha_1,\ldots,\alpha_n)\mid \sum_{i=1}^n\alpha_i=k, \alpha_i\in\N\cup\{0\},i=1,\ldots n\right\}= {{k+n-1}\choose k}\leq Ck^n,
\ee
where $C$ depends on $n$ but not on $k$. Using the decreasing property of $\theta$ in the first integral and the condition (\ref{cond}) in the second integral it follows that 
\beas
D_kf(x) 
&\leq& C(2\pi)^k  k^{n/2} \left( k^{4n} \int_{0}^{k^4}{r^{k-1}e^{-\theta(k^4)r}dr} + \int_{k^4}^{\infty} r^{k+n-1}e^{-2\sqrt{r}}dr \right)\\
&\leq& C(2\pi)^k  k^{n/2} \left( k^{4n} \int_{0}^{\infty}{r^{k-1}e^{-\theta(k^4)r}dr} + e^{-k^2}\int_{0}^{\infty} r^{k+n-1}e^{-\sqrt{r}}dr \right).\\
\eeas
Applying the change of variables $\theta(k^4)r=u$ and $\sqrt{r} = u$ in the respective integrals we obtain  
\beas
D_kf(x) 
&\leq& C(2\pi)^k  k^{\frac{n}{2}} \left( \frac{k^{4n}}{\{\theta(k^4)\}^k} \int_{0}^{\infty}{u^{k-1}e^{-u}du} ~ + ~ 2 e^{-k^2}\int_{0}^{\infty} u^{2k+2n-1}e^{-u} du \right)\\
&\leq& C(2\pi)^k  k^{\frac{n}{2}} \left( \frac{k^{4n} (k-1)!}{\{\theta(k^4)\}^k} ~ + ~ 2 e^{-k^2} (2k+2n-1)! \right)\\
&&  \:\:\:\: (\txt{ as } \Gamma (m)=(m-1)!, ~ m\in\N)\\
&\leq& C(2\pi)^k  \left( k^{4n+\frac{n}{2}-1}  \frac{k!}{\{\theta(k^4)\}^k} ~ + ~ 2k^{\frac{n}{2}} e^{-k^2} (2k+2n-1)! \right)\\
&\leq& C \left\{ \left(\frac{4\pi k}{\theta(k^4)}\right)^k ~ + ~ \left(\frac{4\pi(3k)^3}{e^{k}}\right)^k \right\}.
\eeas 
In the last inequality we have used the trivial estimates $k! \leq k^k$, $2k+2n-1 \leq 3k$ and $k^N \leq 2^{k}$ for fixed $N$ and large enough $k$. Clearly the second term goes to zero as $k$ goes to infinity. Hence, it follows that for large $k$, we have 
\bes
{D_kf(x) \leq C\left\{\frac{4\pi k}{\theta(k^4)}\right\}^k},  \:\:\:\: \txt{ for all $x\in\R^n$.}
\ees 
Applying the change of variable $y = u^4$ in the integral defining $I$, it follows that 
\bes
\int_1^{\infty} \frac{\theta(u^4)}{u}du=\infty.
\ees 
As $\theta$ is decreasing this in turn implies that
\be \label{divsum}
\sum_{k \in \N} \frac{\theta(k^4)}{k}=\infty.
\ee 
If $f$ vanishes on any open set then for any interior point $x_0$ of that set $D_kf(x_0)=0$ for all $k \in \N \cup \{0\}$. Applying Theorem \ref{bochnertaylor} with $m_k = C(4\pi k)^k\theta(k^4)^{-k}$ and using (\ref{divsum}) it follows that $f=0$.

Now we consider the general case, that is, $\theta$ is any nonnegative radial function on $\R^n$ decreasing to zero as $|y| \to \infty$. Since translation of $f$ does not change the absolute value of $\widehat{f}$, without loss of generality we can assume that $f$ vanishes on a ball $B(0,l)$. We consider the function 
\bes
\theta_1(y)= \frac{4}{\sqrt{|y|+1}}, \:\:\:\: y \in \R^n.
\ees
It is clear that the integral $I$ is finite if $\theta$ is replaced by $\theta_1$. Hence, by (b)
there exists a nontrivial $f_1 \in C_c^\infty(\R^n)_0$ such that $\txt{supp}~f_1 \subset B(0,l/2)$ and 
\bes
|\widehat{f_1}(\xi)| \leq Ce^{-\theta_1(\xi)|\xi|}, \:\:\:\: \xi \in \R^n.
\ees 
We now consider the function $h =f * f_1$. We claim that $h$ vanishes on the open set $B(0,l/2)$. Indeed, if $|x|<l/2$ then 
\bes
|h(x)| = \left|\int_{\R^n} f(x-y) f_1(y) dy \right| \leq C \int_{B\left(0,\frac{l}{2}\right)} |f(x-y)|dy = 0,
\ees as $f$ vanishes on $B(0,l)$. Using the inequality
\bes
\theta(\xi) + \theta_1(\xi) \geq 2|\xi|^{-1/2}, \:\:\:\: \txt{ for $|\xi| \geq 1$,}
\ees
and the trivial estimate 
\bes
|\widehat{h}(\xi)|=|\widehat f(\xi)||\widehat f_1(\xi)|\leq Ce^{-(\theta(\xi)+\theta_1(\xi))|\xi|},
\ees 
it follows that $h$ satisfies all the conditions used to prove the special case of (a). Consequently $h(x)=0$ for all $x\in\R^n$. 
This implies that $\widehat{h}(\xi) = \widehat{f}(\xi) \widehat{f_1}(\xi) = 0$ for all $\xi \in \R^n$. As $f_1 \in C_c^\infty(\R^n)$ it follows that $\widehat{f_1}$ extends to an entire function on $\C^n$. In particular, $\widehat{f_1}$ is a real analytic function on $\R^n$ and so it cannot vanish on a set of positive measure in $\R^n$. So $\widehat{f}(x)=0$ for almost every $x\in\R^n$. This completes the proof.
\end{proof}

We now present a higher dimensional version of Theorem \ref{paleywiener}.
\begin{thm}\label{halfspace}
Let $\psi:\R^{n}\rightarrow[0,\infty)$ be a locally integrable radial function such that 
\be \label{condhalf} 
I = \int_{\R^n}{\frac{\psi(x)}{(1+|x|)^{n+1}}dx}.
\ee 
\begin{enumerate}
\item[(a)] Let $f\in L^2(\R^n)$ be a nonzero function satisfying the estimate
\be \label{ingcond3}
|\widehat{f}(y)|\leq Ce^{-\psi(y)},  \:\:\:\: \txt{ for almost every $y\in\R^n$.}
\ee
Suppose $\txt{supp}~f \subset \{x\in\R^n \mid x\cdot\eta \leq t\}$ for some $\eta\in S^{n-1}$ and some $t\in \R$. If $I=\infty$ then $f(x) = 0$ for almost every $x\in\R^n$. 
\item[(b)] If $I$ is finite and $\psi$ is nondecreasing then given any $l>0$ there exists a nontrivial $f \in C_c^{\infty}(\R^n)$ supported in the  ball $B(0,l)$ satisfying (\ref{ingcond3}).
\end{enumerate}
\end{thm} 

\begin{proof} 
We define $g(x)=f(x+t\eta)$ for almost every $x \in \R^n$. Since 
\bes 
(x+t\eta) \cdot \eta = x \cdot \eta + t > t, \:\:\:\: \txt{ for } x \cdot \eta>0,
\ees 
we get that $\txt{supp}~g \subset \{x\in\R^n \mid x \cdot \eta \leq 0\}$. Consider $T\in SO(n)$ such that $T^t\eta=e_1$ where $e_1 = (1,0, \cdots, 0)$ and define $h(x)=g(Tx)$ for almost every $x \in \R^n$. Since $\langle Tx,\eta \rangle = \langle x, e_1 \rangle$, we get that $\txt{supp}~h \subset \{x\in\R^n \mid x \cdot e_1 \leq 0\}$. Moreover, for almost every $y \in \R^n$,
\bes 
|\widehat h(y)|=|\widehat g(Ty)| = |\widehat{f}(Ty)|.
\ees 
Since $\psi$ is radial, it is thus enough to prove the result for $f\in L^2(\R^n)$ with $\txt{supp}~f \subset \{x\in \R^n \mid x \cdot e_1 \leq 0\}$. 

Since $\psi$ is radial and locally integrable, we can consider $\psi$ as a locally integrable even function on $\R$ and then the given condition is equivalent to 
\bes 
\int_M^\infty \frac{\psi(r)}{1+r^2}dr=\infty, \:\:\:\: \txt{ for large } M.
\ees 
For a fixed $y \in \R^{n-1}$, we define a locally integrable function $\psi_y$ on $\R$ by $\psi_y(x)=\psi(x,y)$ for almost every $x \in \R$. Note that $y \in \R^{n-1}$ is taken from a full measure set as $\psi$ is defined pointwise almost everywhere on $\R^n$. Now, for fixed $y \in \R^{n-1}$ using the change of variable $r^2 = x^2 + |y|^2$ it follows that
\beas 
\int_M^\infty \frac{\psi_y(x)}{1+x^2}dx 
&\geq& \int_M^\infty \frac{\psi \left(\sqrt{x^2+|y|^2}\right)}{1+x^2+|y|^2}dx \\
&=& \int_{\sqrt{M^2+|y|^2}}^\infty \frac{\psi(r)}{1+r^2}\frac{r}{\sqrt{r^2-|y|^2}}dr\\
&\geq& \int_{\sqrt{M^2+|y|^2}}^\infty \frac{\psi(r)}{1+r^2}dr \\
&=& \infty.
\eeas
We define $g_y(x) = \mathcal F_{n-1}f(x,y)$ for almost every $x \in \R$ where $\mathcal F_{n-1}$ denotes the $(n-1)$ dimensional Fourier transform of the function $u \mapsto f(x,u), ~ u \in \R^{n-1}$. Since $f \in L^2(\R^n)$, by Plancherel Theorem it follows that $g_y \in L^2(\mathbb {R})$ and 
\bes
g_y(x)=0, \:\:\:\: \txt{ for } x>0.
\ees 
It follows that 
\bes 
\widehat g_y(\xi)= \mathcal F_nf(\xi,y), \:\:\:\: \txt{ for almost every } \xi \in \R,
\ees 
where $\mathcal F_n$ denotes the $n$ dimensional Fourier transform of the function $f$ and $|\widehat g_y(\xi)| = O(e^{-\psi_y(\xi)})$ as $|\xi| \ra \infty$. By Theorem \ref{paleywiener} we conclude that $g_y=0$. Since this is true for almost every $y\in{\R^{n-1}}$ we get that $f=0$. This proves (a).

To prove (b) we observe that by  Lemma 4 of \cite{L2} there exists $g_1\in C_c(\R)$ such that 
\bes
|\widehat g_1(\xi)| \leq e^{-\psi(\xi)}, \:\:\:\: \txt{ for all } \xi \in \R^n.
\ees 
We now use Radon transform and proceed exactly as in Theorem \ref{open} to construct the required function $f \in C_c^{\infty}(\R^n)$.
\end{proof}

The following is a simple corollary of Theorem \ref{halfspace}. 
\begin{cor}
Let $\psi$, $I$ be as in Theorem \ref{halfspace} and suppose $T$ is a compactly supported distribution on $\R^n$ such that 
\be \label{distribution}
|\widehat{T}(x)|\leq Ce^{-\psi(x)}, \:\:\:\:\txt{ for all $x\in\R^n$.}
\ee
If $I=\infty$ then $T=0$.
\end{cor}
\begin{proof}
We consider $\phi\in C_c^{\infty}(\R^n)$ and define $\tilde{\phi} \in  C_c^{\infty}(\R^n)$ by
\bes
\tilde{\phi}(x) = \phi(-x), \:\:\:\:\txt{ for $x\in\R^n$.}
\ees
Then $T\ast\tilde{\phi} \in C_c^{\infty}(\R^n)$ and it's Fourier transform satisfies (\ref{distribution}). If $I=\infty$ then it follows from Theorem \ref{halfspace} that $T\ast\tilde{\phi}$ is zero. Therefore 
\bes
T(\phi) = (T \ast \tilde{\phi})(0) = 0.
\ees
Since this is true for any $\phi\in C_c^{\infty}(\R^n)$ we get that $T$ is zero.
\end{proof}

It is easy to see that one can use the same technique as in Theorem \ref{halfspace} to prove the following.
\begin{cor}
Let $\psi:\R^n\rightarrow[0,\infty)$ be a locally integrable function such that for almost every $x' \in \R^{n-1}$, 
\bes 
\int_{\R}{\frac{\psi(x,x')}{1+x^2} dx} = \infty.
\ees 
If $f\in L^2(\R^n)$ is a nonzero function with $\txt{supp}~f \subset \{x\in\R^n \mid x\cdot e_1 \leq 0\}$ such that $|\widehat f(y)|=O(e^{-\psi(y)})$ pointwise almost everywhere as $|y| \ra \infty$, then $f = 0$. 
\end{cor}

\begin{rem}
\begin{enumerate}
\item
If we consider $f \in C_c(\R^n)$ in the first part of Theorem \ref{halfspace}, there is a simpler proof using Radon transform. To be precise, for fixed $\omega\in S^{n-1}$ we apply Theorem \ref{paleywiener} in the introduction to the compactly supported function $Rf(\omega,\cdot)$ on $\R$ and use the slice projection theorem (\ref{sliceproj}) to get that $f=0$.
\item
If $\psi$ is not assumed to be a radial function then $I = \infty$ in (\ref{condhalf}) does not imply that $f=0$. We illustrate this by the following example. By Theorem \ref{paleywiener} there exists a nonzero $f \in L^2(\R)$ supported in the set $\{x\in\R\mid x\leq x_0\}$ for some $x_0 \in \R$ such that 
\bes
|\widehat{f}(y)|\leq Ce^{-|y|^{1/2}}, \:\:\:\: \txt{ for all large } y \in \R.
\ees 
We define $F(x,y) = f(x) P_1(y)$, for almost every $(x,y) \in \R^2$ where 
\bes 
P_{\alpha}(x) = \frac{1}{\pi} \frac{\alpha}{\alpha^2+x^2}, \:\:\:\: \txt{ for } \alpha>0, ~ x \in \R,
\ees 
denotes the standard Poisson kernel of the upper half plane $\{ (x,y)\in\R^2\mid x \in \R, y >0 \}$. Clearly $F \in L^2(\R^2)$ with $\txt{supp}~F\subset\{ (x,y)\in\R^2\mid x \leq x_0 \}$ and 
\bes 
|\widehat{F}(u,v)| = |\widehat{f}(u)| |\widehat{P_1}(v)| \leq Ce^{-\psi(u,v)}, \:\:\:\: \txt{ for almost every } (u,v) \in \R^2, 
\ees 
where $\psi(u,v) = |u|^{1/2} + 2\pi|v|$. However, it is easy to see that the integral 
\bes
\int_{\R^2} \frac{\psi(u,v)}{(1 + |(u,v)|)^3} dudv = \int_{0}^{\infty} \int_{0}^{2\pi} \frac{r^{\frac{1}{2}} |\cos \theta|^{\frac{1}{2}} + r |\sin \theta|^{\frac{1}{2}}}{(1+r)^3} r dr d\theta = \infty.
\ees
\end{enumerate}
\end{rem}
 
\subsection{Ingham's theorem for the torus $\T^n$}
In this subsection our aim to prove an analogue of Theorem \ref{open} for the torus $\T^n=\{(e^{2\pi ix_1},\ldots,e^{2\pi ix_n})\in\C^n\mid (x_1,\ldots ,x_n)\in [0,1]^n\}$. Consequently we can identify $\T^n$ with the set $[0,1]^n=\{x=(x_1,\ldots ,x_n)\mid x_j\in [0,1], 1\leq j\leq n\}$. If $f\in L^1(\T^n)$ then we define its Fourier transform $\widehat{f}$ by the formula
\bes
\widehat{f}(m)=\int_{\T^n}f(x)e^{-2\pi im \cdot x}dx, \:\:\:\:m\in\Z^n.
\ees
If $\theta$ is a radial function on $\R^n$ then it is basically a function of $|x|$ for all $x\in\R^n$ and hence can be thought of as a function on $[0,\infty)$. In the following we are going to restrict $\theta$ on the set of natural numbers $\N$ which is well defined by the preceding observation. 

\begin{thm}\label{torus}
Let $\theta:\R^n\to [0,\infty)$ be a decreasing radial function with $\lim_{|y|\to\infty}\theta (y)=0$ and
\bes
S=\sum_{m\in \N} \frac{\theta(m)}{m}.
\ees
\begin{enumerate}
\item[(a)] Let $f\in L^1(\T^n)$ be such that 
\be \label{ingdecay-torus}
|\widehat f(m)|\leq Ce^{-\theta(m)|m|}, 
\ee
for all $m\in\Z^n$.
If $f$ vanishes on any nonempty open set $U\subset \T^n$ and $S=\infty$ then $f(x)=0$ for all $x\in\T^n$.
\item[(b)] If $S$ is finite then there exists $f \in L^1(\T^n)$ supported on any given open set $U \subset \T^n$ satisfying (\ref{ingdecay-torus}).  
\end{enumerate}
\end{thm}

\begin{proof}
We shall first prove (b). Since $\theta$ is a decreasing function it follows from the hypothesis about $S$ in (b) that
\bes
\int_{1}^{\infty}{\frac{\theta(y)}{y}dy}<\infty.
\ees 
Hence by Theorem \ref{open}, given any $\epsilon >0$, there exists a radial function $g\in C_c^{\infty}(\R^n)$ with $\txt{supp}~g \subset B(0,\epsilon)$ and 
\bes |\widehat g(\xi)| \leq Ce^{-\theta(\xi)|\xi|}, \:\:\:\: \txt{ for all } \xi\in\R^n.
\ees 
We now define 
\bes
f(x)=\sum_{m\in\Z^n}g(x+m), \:\:\:\:x \in \T^n.
\ees
Then $f \in L^1(\T^n)$ and by the Poisson summation formula (Theorem 2.4, \cite{SW})
\be \label{torusdecay}
|\widehat{f}(m)| = |\widehat g(m)| \leq Ce^{-\theta(m)|m|}, \:\:\:\: m\in\Z^n.
\ee
Moreover, (\ref{torusdecay}) is also true for any translate of $f$. So, we can choose $\epsilon>0$ in such a way that some translate of $f$ is supported inside the given open set $U \subset \T^n$. This completes the proof of (b). To prove (a) we start with the special case 
\be \label{condT}
\theta(m) \geq \frac{2}{\sqrt{|m|}}, \:\:\:\:\txt{ for all large $|m|$},~ m \in \mathbb Z^n.
\ee 
It follows from (\ref{ingdecay-torus}) that $\widehat{f}$ satisfies the estimate  
\bes
|\widehat f(m)| \leq C e^{-2\sqrt{|m|}}, \:\:\:\:\txt{ for all $m\in\Z^n$.}
\ees
Hence by the Fourier inversion $f \in C^\infty(\T^n)$ and 
\beas
D_kf(x) 
&\leq& (2\pi)^k \left\{\sum_{|\alpha|=k} {\left(\sum_{m\in\Z^n}{|m_1|^{\alpha_1} \cdots |m_n|^{\alpha_n}|\widehat f(m)|}\right)^2}\right\}^{\frac{1}{2}}\\
&\leq& (2\pi)^k \left\{\sum_{|\alpha|=k} \left(\sum_{m\in\Z^n}|m|^k |\widehat f(m)|\right)^2\right\}^{\frac{1}{2}}\\
&\leq& C  (2\pi)^k k^{\frac{n}{2}}\sum_{m\in \Z^n}{|m|^ke^{-\theta(m)|m|}}\\
&=& C (2\pi)^k k^{\frac{n}{2}}\sum_{p\in\N} \sum_{|m|^2=p}|m|^k e^{-\theta(m)|m|}\\
&\leq& C  (2\pi)^k k^{\frac{n}{2}} \sum_{p\in\N} p^{\frac{k}{2}+n} e^{-\theta(\sqrt p)\sqrt p}\\
&=& C  (2\pi)^k k^{\frac{n}{2}}\left( \sum_{p=1}^{k^8-1}{p^{{\frac{k}{2}}+n}e^{-\theta(\sqrt p)\sqrt p}}+ \sum_{p=k^8}^{\infty} {p^{{\frac{k}{2}}+n}e^{-\theta(\sqrt p)\sqrt p}}\right).
\eeas
In the above we have used the estimate (\ref{count}) and the fact that 
\bes
\#\left\{(m_1,\ldots,m_n)\mid \sum_{i=1}^nm_i^2=p, m_i\in\Z,i=1,\ldots n\right\}\leq Cp^n,
\ees
where $C$ is independent of $p$.
We now use the decreasing property of $\theta$ in the first expression and (\ref{condT}) in the second expression to get that for all large $k$,  
\beas
D_k f(x)
&=& C (2\pi)^k  k^{\frac{n}{2}}\left( \sum_{p=1}^{k^8-1}p^{\frac{k}{2}+n}e^{-\theta(\sqrt p)(\sqrt p+1-1)}+ \sum_{p=k^8}^{\infty} {p^{{\frac{k}{2}}+n}e^{-\theta(\sqrt p)\sqrt p}}\right)\\
&\leq& C (2\pi)^k k^{\frac{n}{2}}\left(e^{\theta(1)} \sum_{p=1}^{k^8-1}p^{\frac{k}{2}+n}e^{-\theta(\sqrt p)(\sqrt p+1)}+ \sum_{p=k^8}^{\infty} {p^{{\frac{k}{2}}+n}e^{-\theta(\sqrt p)\sqrt p}}\right)\\
&\leq& C (2\pi)^k k^{\frac{n}{2}} \left( \sum_{p=1}^{k^8-1}{p^{{\frac{k}{2}}+n}e^{-\theta(\sqrt p)\sqrt{p+1}}} + \sum_{p=k^8}^{\infty}{p^{{\frac{k}{2}}+n}e^{-2p^{\frac{1}{4}}}} \right)\\
&\leq& C (2\pi)^k k^{\frac{n}{2}} \left( \sum_{p=1}^{k^8-1}{\int_{p}^{p+1}{y^{{\frac{k}{2}}+n}e^{-\theta(\sqrt p)\sqrt y}dy}} + \sum_{p=k^8}^{\infty}{\int_{p}^{p+1}{y^{{\frac{k}{2}}+n}e^{-2(y-1)^{\frac{1}{4}}}}dy} \right)\\ 
&\leq& C (2\pi)^k k^{\frac{n}{2}} \left( \int_{1}^{k^8}{y^{{\frac{k}{2}}+n}e^{-\theta(k^4)\sqrt y}dy} + \int_{k^8-1}^{\infty}{(y+1)^{{\frac{k}{2}}+n}e^{-y^{\frac{1}{4}}}e^{-y^{\frac{1}{4}}}dy} \right) \\
\eeas
\beas
&\leq& C (2\pi)^k k^{\frac{n}{2}}\left(k^{8(n+1)}\int_{0}^{\infty}y^{\frac{k}{2}-1}e^{-\theta(k^4)\sqrt y}dy + e^{-(k^8-1)^{\frac{1}{4}}} \int_{k^8-1}^{\infty}(2y)^{\frac{k}{2}+n}e^{-y^{\frac{1}{4}}}dy\right)\\
&\leq& C (2\pi)^k k^{\frac{n}{2}} k^{8(n+1)}\int_{0}^{\infty}y^{\frac{k}{2}-1}e^{-\theta(k^4)\sqrt y}dy + C (2\pi)^k k^{\frac{n}{2}} 2^{k} e^{-\frac{k^2}{2}} \int_{0}^{\infty}y^{\frac{k}{2}+n}e^{-y^{\frac{1}{4}}}dy.
\eeas
In the above we have used the trivial estimates 
\beas 
y+1 &\leq& 2y, \:\:\:\: \txt{ for } y \geq k^8-1, \\ 
(k^8-1)^{1/4} &\geq& k^2/2, \:\:\:\: \txt{ for large } k.
\eeas
Applying change of variables $\sqrt{y}=u$ and $y^{\frac{1}{4}} = u$ in the respective integrals we obtain that for all large $k$,  
\beas 
D_kf(x)
&\leq& C (2\pi)^k k^{8n+{\frac{n}{2}}+8}\int_{0}^{\infty}{u^{k-1}e^{-\theta(k^4)u}du} + C k^{\frac{n}{2}} (4\pi)^{k} e^{-\frac{k^2}{2}} \int_{0}^{\infty}{u^{2k+4n+3}e^{-u}du}\\
&=& C (2\pi)^k \frac{k^{8n+\frac{n}{2}+8} k!}{\left\{\theta(k^4)\right\}^k} + Ck^{\frac{n}{2}}(4\pi)^{k} e^{-\frac{k^2}{2}}(2k+4n+3)!\\
&\leq& C \left(\frac{4\pi k}{\theta(k^4)}\right)^k + C\left(\frac{8\pi 3^3 k^3}{e^{k/2}}\right)^k, \:\:\:\:  \txt{ for large } k.
\eeas
In the above we have used the trivial estimates 
\bes
k^{8n+\frac{n}{2}+8} \leq 2k, \:\: k! \leq k^k, \:\: 2k+4n+3 \leq 3k,
\ees
for all large $k$. Clearly the second term goes to zero as $k$ goes to infinity. It follows that for large $k$, we have 
\bes
{D_kf(x) \leq C\left\{\frac{4\pi k}{\theta(k^4)}\right\}^k},  \:\:\:\: \txt{ for all $x\in\R^n$.}
\ees 
Since $\theta$ is decreasing and $S$ is infinite it follows that 
\bes
\int_{1}^{\infty}\frac{\theta(y)}{y}dy=\infty.
\ees 
From this it follows exactly as in the proof of Theorem \ref{open} that
\bes
\sum_{k=1}^{\infty}{\frac{\theta(k^4)}{k}}=\infty. 
\ees
By Theorem \ref{bochnertaylor} we now conclude that $f(x)=0$ for all $x\in\T^n$. We now take care of the general case. Consider the non negative radially decreasing function $\theta_1$ on $\R^n$ given by 
\bes 
\theta_1(x)=\frac{4}{\sqrt{|x|+1}}, \:\:\:\: \txt{ for }x \in \R^n.
\ees 
Clearly 
\bes
\sum_{m \in \N}\frac{\theta_1(m)}{m}< \infty,
\ees
and hence by (b) there exists $g \in L^1(\T^n)$ such that $\txt{supp}~g$ is contained in a small open set $U$ in $\T^n$ and 
\be \label{decayf1}
|\widehat{g}(m)| \leq Ce^{-\theta_1(m)|m|}, \:\:\:\:\txt{ for all $m\in\Z^n$.}
\ee 
If we define $F(x)=f*g(x)$ for $x \in \T^n$, then by choosing $U$ appropriately it follows exactly as in Theorem \ref{open} that $F$ vanishes on an open set in $\T^n$. Moreover, $\widehat{F}$ satisfies the estimate 
\bes
|\widehat F(m)| \leq Ce^{-|m|\left\{\theta(m)+\theta_1(m)\right\}}, \:\:\:\:m\in\Z^n.
\ees
Since, 
\bes
\theta(m)+\theta_1(m) \geq \frac{4}{\sqrt {|m|+1}}, \:\:\:\: m\in\Z^n,
\ees
it follows from the special case proved already that $F(x)=0$ for all $x\in\T^n$. Consequently, 
\bes
\widehat f(m) \widehat g(m)=0, \:\:\:\: \txt{ for all }m\in \Z^n.
\ees 
Since $g$ is nonzero there is a $m_0\in \Z^n$ such that $\widehat g(m_0)$ is nonzero. This implies that $\widehat f(m_0)$ is zero. 
Our aim is to prove that 
\bes 
\widehat{f}(m) = 0, \:\:\:\: \txt{ for every } m \in \Z^n.
\ees 
We now fix $m \in \Z^n$ and define a function $g_m$ on $\T^n$ by
\bes 
g_m(x)=e^{2\pi i(m-m_0) \cdot x}g(x), \:\:\:\: \txt{ for } x\in\T^n,
\ees 
so that 
\bes \widehat g_m(l) = \widehat g(l+m_0-m), \:\:\:\: \txt{ for all } l \in \Z^n.
\ees 
Defining 
\bes 
F_m(x)=f*g_m(x), \:\:\:\: \txt{ for } x \in \T^n,
\ees 
and using (\ref{ingdecay-torus}), (\ref{decayf1}) and the fact that $\theta_1$ is decreasing, we get that for any $l \in \Z^n$,
\beas
|\widehat F_m(l)| 
&=& |\widehat f(l)\widehat g(l+m_0-m)|\\
&\leq& C e^{-|l|\theta(l)} e^{-|l+m_0-m|\theta_1(l+m_0-m)}\\
&\leq& C e^{-|l|\theta(l)} e^{-|l|\theta_1(l+m_0-m)}e^{|m_0-m|\theta_1(0)}\\
&\leq& C e^{-|l|\theta(l)} e^{-\frac{4|l|}{\sqrt{|l+m_0-m|+1}}}.
\eeas
Now, for large $|l|$, the quantity $|l+m_0-m|+1$ can be dominated by $4|l|$. Hence
\bes 
|\widehat F_m(l)| \leq C e^{-\left(\frac{2}{\sqrt{|l|}} + \theta(l)\right)|l|}.
\ees 
Using translation as before, we can make $F_m$ vanish on an open set in $\T^n$. It follows from the special case proved before that $F_m$ vanishes identically. Since $\widehat g_m(m) = \widehat g(m_0)$ is nonzero we conclude that $\widehat{f}(m)$ is zero. As $m \in \Z^n$ was arbitrary it follows that $f=0$.

\end{proof}

\section{Step Two Nilpotent Lie Groups}

In this section we present analogues of the results of Ingham and Paley-Wiener on connected, simply connected two step nilpotent Lie groups. We start with the required preliminaries on two step nilpotent Lie groups.

\subsection{Representations of Step Two Nilpotent Lie Groups}

A complete account of representation theory for general connected, simply connected nilpotent Lie groups can be found in \cite{CG}. Representations of step two connected, simply connected nilpotent Lie groups and the Plancherel theorem is described in \cite{ACDS, R, PS1}. We briefly describe the basic facts to make the paper self contained.

Let $G$ be a step two connected, simply connected nilpotent Lie group. Then its Lie algebra $\mathfrak g$ has the decomposition $\mathfrak g=\mathfrak v\oplus \mathfrak z$, where 
\bes
\{0\} \neq [\mathfrak g, \mathfrak g] = [\mathfrak v, \mathfrak v] \subset \mathfrak z, \:\:\:\: [\mathfrak g, \mathfrak z] = \{ 0\}.
\ees 
Since $G$ is nilpotent the exponential map from $\mathfrak g$ to $G$ is an analytic diffeomorphism. We can identify $G$ with $\mathfrak v\oplus \mathfrak z$ and write $(V+Z)$ for $\exp(V+Z)$ and denote it by $(V,Z)$ where $V\in \mathfrak v$ and $Z\in \mathfrak z$. The product law on $G$ is given by the Baker-Campbell-Hausdorff formula
\bes 
(V,Z)(V',Z')=\left(V+V',Z+Z'+\frac{1}{2}{[V,V']}\right), \: \:\:\: \txt{ for all } V,V'\in \mathfrak v, \:\: Z,Z'\in \mathfrak z.
\ees 

Let $\mathfrak g^*$ and $\mathfrak z^*$ be the dual vector spaces of $\mathfrak g$ and $\mathfrak z$ respectivly. For each $\nu\in \mathfrak z^*$ we consider the bilinear form $B_\nu$ on $\mathfrak v$ defined by 
\bes 
B_\nu(V,V')=\nu([V,V']), \:\:\:\: \txt{ for all }V, V'\in \mathfrak v.
\ees 
Let 
\bes 
\tau_\nu=\{V \in \mathfrak v \mid \nu([V,V'])=0, \txt{ for all } V' \in \mathfrak v\}.
\ees 
We fix an orthonormal basis $\mathcal B = \{ V_1, V_2, \cdots, V_m, Z_1, \cdots, Z_k \}$ of $\mathfrak g$ such that 
\beas
\mathfrak v &=& \txt{span}_{\R} \{ V_1, \cdots , V_m \},\\ \mathfrak z &=& \txt{span}_{\R}\{Z_1, \cdots, Z_k\}.
\eeas 
Let $m_\nu$ be the orthogonal complement of $\tau_\nu$ in $\mathfrak v$. Then $\mathcal U=\{ \nu\in \mathfrak z^*\mid \dim(m_\nu) \txt{ is maximum}\}$ is a Zariski open subset of $\mathfrak z^*$. We shall denote by $S_{\nu}$ the $m \times m$ matrix whose $(i,j)$-th entry is  $B_{\nu}(V_i, V_j)$ for $1 \leq i,j \leq m$. 
\begin{defn}
The index $i$ is called a jump index for $\nu \in \mathfrak z^*$ if the rank of the $i \times m$ submatrix of $S_{\nu}$, consisting of
the first $i$ rows, is strictly greater than the rank of the $(i-1) \times m$ submatrix of $S_{\nu}$, consisting
of the first $i-1$ rows.
\end{defn}
\noindent Since $B_\nu$ is an alternating bilinear form, the number of jump indices is even. The jump indices depend on $\nu$ and on the order of the basis as well. However $\nu\in\mathcal U$ have the same set of jump indices which we shall denote by 
\bes 
P=\{j_1, j_2, \cdots, j_{2n}\} \subset \{1,2, \cdots ,m\}.
\ees 
Let $Q=\{n_1,n_2,\cdots, n_r\}$ be the complement of $P$ in $\{1,2, \cdots ,m\}$. Let 
\bes 
A_P=\txt{span}_{\R}\{V_{j_1}, \cdots, V_{j_{2n}}\},
\ees 
\beas 
A_Q &=& \txt{span}_\R\{Z_1,\cdots, Z_k,V_{n_i} \mid n_i\in Q\}, \:\:\:\: \widetilde{A}_Q =\txt{span}_{\R}\{V_{n_i} \mid n_i\in Q\},\\ 
A^*_Q &=& \txt{span}_{\R}\{Z_1^*,\cdots, Z_k^*,V^*_{n_i} \mid n_i\in Q\}, \:\:\:\: \widetilde {A}^*_Q=\txt{span}_{\R}\{V^*_{n_i} \mid n_i\in Q\},
\eeas 
where $\mathcal B^* = \{ V^*_1, V^*_2, \cdots, V^*_m, Z^*_1, \cdots, Z^*_k\}$ is the dual basis of $\mathcal B$. The irreducible unitary representations relevant to the Plancherel measure are parametrized by the set $\Lambda=\widetilde A^*_Q\times\mathcal U$. 
\begin{defn}
Let $G$ be a connected, simply connected two step nilpotent Lie group. If there exists $\nu\in \mathfrak z^*$ such that $B_\nu$ is nondegenerate then we call $G$ a step two nilpotent Lie group with MW condition or step two MW group. 
\end{defn}

In the following two subsections, we shall give an explicit description of the group Fourier transform of a function defined on $G$ and the Plancherel measure. These are going to be crucial for the proof of Theorem \ref{openG} and Theorem \ref{halfspG}. 

\subsubsection{Step two non MW groups}

In this case $\tau_\nu\neq\{0\}$ for each $\nu\in\mathcal U$. Then $B_\nu|_{m_\nu}$ is nondegenerate and hence dim $m_\nu=2n$. From the properties of an alternating bilinear form there exists an orthonormal basis $\{X_1(\nu),Y_1(\nu),\cdots,X_n(\nu),Y_n(\nu),U_1(\nu),\cdots, U_r(\nu)\}$ of $\mathfrak v$ and positive numbers $d_i(\nu)>0$ such that
\begin{enumerate}
\item  $\tau_\nu=\txt{span}_{\R}\{U_1(\nu),\cdots, U_r(\nu)\} $,
\item $\nu([X_i(\nu),Y_j(\nu)])=\delta_{i,j}d_j(\nu),1\leq i,j\leq n$.
\end{enumerate}
We call the basis 
\bes 
\{X_1(\nu),\cdots,X_n(\nu),Y_1(\nu),\cdots,Y_n(\nu),U_1(\nu),\cdots,U_r(\nu),Z_1,\cdots,Z_k\}
\ees 
of $\mathfrak g$ an almost symplectic basis. Let 
\beas 
\xi_\nu &=& \txt{span}_{\R}\{X_1(\nu),\cdots,X_n(\nu)\}, \\
\eta_\nu &=& \txt{span}_\R\{Y_1(\nu),\cdots,Y_n(\nu)\}.
\eeas
Then we have the decomposition 
\bes
\mathfrak g = \xi_\nu\oplus\eta_\nu\oplus\tau_\nu\oplus \mathfrak z = \{ X+Y+U+Z \mid X\in \xi_\nu, Y\in\eta_\nu, U\in\tau_\nu, Z\in \mathfrak z \}.
\ees 
We shall denote the element $\exp(X+Y+U+Z)$ of $G$ by $(X,Y,U,Z)$. We shall write 
\bes 
(X,Y,U,Z)=\sum_{j=1}^{n}{x_j(\nu)X_j(\nu)}+\sum_{j=1}^{n}{y_j(\nu)Y_j(\nu)}+\sum_{j=1}^{r}{u_j(\nu)U_j(\nu)+\sum_{j=1}^{k}{z_jZ_j}}
\ees and denote it by $(x,y,u,z)$ suppressing the dependence of $\nu$ which will be understood from the context. If we take $\lambda\in\Lambda$ then it can be written as $\lambda=(\mu,\nu)$, where $\mu\in\widetilde A^*_Q$ and $\nu\in\mathcal U$. Therefore 
\bes 
\lambda=(\mu,\nu)=\sum_{i=1}^{r}{\mu_iV^*_{n_i}}+\sum_{i=1}^{m}{\nu_iZ^*_i}.
\ees 
We can extend $\lambda$ to a linear functional $\lambda'$ on $\mathfrak g$ simply by defining it to be zero on $A_P$. We define 
\bes 
\tilde\mu_i = \lambda'(U_i(\nu)), \:\:\:\: \txt{ for } 1 \leq i \leq r
\ees
and consider the linear map 
\be \label{M_nu} 
M_{\nu}:\widetilde A^*_Q\rightarrow  \txt{span}_\R\{U_1(\nu)^*, \cdots ,U_r(\nu)^*\} 
\ee 
given by 
\bes 
M_{\nu}(\mu_1,\cdots,\mu_r) =(\tilde \mu_1,\cdots,\tilde \mu_r).
\ees 
It has been shown in \cite{R} that 
\bes 
|\det (J_{M_{\nu}})| = \frac{|Pf(\nu)|}{d_1(\nu)\cdots, d_n(\nu)},
\ees 
where $J_{M_{\nu}}$ is the Jacobian matrix of $M_{\nu}$ and $Pf(\nu)$ is the Pfaffian of $\nu$ given by 
\bes
Pf(\nu) = \sqrt{\det(B_{\nu}')}, \:\:\:\: (B_{\nu}')_{is} = \nu([V_{j_i}, V_{j_s}]), \:\:\:\: V_{j_i}, V_{j_s} \in A_P.
\ees 
Using the almost symplectic basis we now describe an irreducible unitary representation $\pi_{\mu,\nu}$ of G realized on $L^2(\eta_\nu)$ by the following action
\bes 
(\pi_{\mu,\nu}(x,y,u,z)\phi)(\xi)= e^ {2\pi i \left( \sum_{j=1}^{k}{\nu_jz_j}+ \sum_{j=1}^{r}{\tilde\mu_j u_j}+ \sum_{j=1}^{n}{d_j(\nu)\left(x_j\xi_j+ \frac{1}{2}x_jy_j\right)} \right)} \phi(\xi+y),
\ees 
for all $\phi\in L^2(\eta_\nu)$. We define the Fourier transform of $f\in L^1(G)$ by the operator valued integral
\bes 
\pi_{\mu,\nu}(f)=\int_{\mathfrak z} \int_{\tau_\nu}\int_{\eta_\nu} \int_{\xi_\nu}{f(x,y,u,z)\pi_{\mu,\nu}(-x,-y,-u,-z)~dx~dy~du~dz},
\ees 
for $\lambda=(\mu,\nu)\in \Lambda$. For $\nu\in \mathfrak z^*$ we consider the Euclidean Fourier transform of $f$ in the central variable given by  
\be \label{centralft}
f^\nu(x,y,u)=\int_{\mathfrak z} { e^{-2 \pi i\sum_{j=1}^{k}{\nu_j z_j}} f(x,y,u,z)dz}.
\ee 
We also define for $\tilde \mu\in\tau^*_\nu$
\bes 
f^{\tilde\mu,\nu}(x,y)=\int_{\tau_\nu}\int_{\mathfrak z} {e^{-2\pi i\sum_{j=1}^{k}{\nu_j z_j}-2\pi i\sum_{j=1}^{r}{\tilde\mu_j u_j}} f(x,y,u,z)dzdu}.
\ees
If $f\in L^1(G) \cap L^2(G)$ then $\pi_{\mu,\nu}(f)$ is an Hilbert-Schmidt operator and we have (see \cite{R}) 
\bes 
\prod_{j=1}^{n} {d_j(\nu)\|\pi_{\mu,\nu}(f)\|^2_{HS}}=\int_{\eta_\nu}\int_{\xi_\nu}{|f^{\tilde \mu,\nu}(x,y)|^2 dx dy}.
\ees 
Here and elsewhere $\|\cdot\|_{HS}$ denotes the Hilbert-Schmidt norm. Now integrating both sides on $\widetilde A^*_Q$ with respect to the Lebesgue measure and applying the transformation given by the function $M_{\nu}$ in (\ref{M_nu}) we get 
\be \label{partplancherel}
|Pf(\nu)|\int_{\tilde A^*_Q}{\|\pi_{\mu,\nu}(f)\|^2_{HS}~d\mu} = \int_{\tau^*_\nu}\int_{\eta_\nu}\int_{\xi_\nu}{|f^{\tilde\mu,\nu}(x,y)|^2~dx~dy~d\tilde\mu} = \int_{\mathfrak v} |f^\nu(v)|^2dv.
\ee
Using the Euclidean Plancherel theorem on the center $\mathfrak z$ we get the Plancherel formula for $G$ given by 
\bes 
\int_\Lambda{\|\pi_{\mu,\nu}(f)\|^2_{HS} |Pf(\nu)|~d\mu~d\nu} = \int_G{|f(v,z)|^2 dv~dz}.
\ees
The above holds for all $L^2$-functions on $G$ by a standard density argument.

\subsubsection{Step two MW groups}

In this case there exists $\nu \in \mathfrak z^*$ such that $B_\nu$ is nondegenerate. So $\mathcal U =\{\nu\in \mathfrak z^*:B_\nu \txt{ is nondegenerate}\}$ and the representations are parametrized by the Zariski open set $\Lambda= \mathcal U$. The representations are given by
\bes 
\left(\pi_\nu(x,y,z)\phi\right)(\xi)= e^{2\pi i\sum_{j=1}^{k}{\nu_j z_j}+2\pi i\sum_{j=1}^{n}{d_j(\nu)\left(x_j\xi_j+\frac{1}{2}x_jy_j\right)}} \phi(\xi+y), \:\:\:\: \phi \in L^2(\eta_\nu).
\ees 
It turns out that in this case 
\bes 
|Pf(\nu)|=\prod_{j=1}^n {d_j(\nu)}, \:\:\:\: \nu \in \Lambda.
\ees 
We define the Fourier transform of $f\in L^1(G)$ by 
\bes 
\pi_{\nu}(f)=\int_\mathfrak z\int_{\eta_\nu} \int_{\xi_\nu}{f(x,y,z)\pi_\nu(-x,-y,-z)~dx~dy~dz}, \:\:\:\: \nu \in \Lambda.
\ees 
If $f\in L^1(G) \cap L^2(G)$ then $\pi_{\nu}(f)$ is a Hilbert-Schmidt operator and 
\be \label{hs} 
|Pf(\nu)|\|\pi_{\nu}(f)\|^2_{HS}=\int_{\eta_\nu}\int_{\xi_\nu}{|f^\nu(x,y)|^2~dx~dy}=\int_v{|f^\nu(x,y)|^2~dx~dy},
\ee 
where $f^{\nu}$ is the Euclidean Fourier transform of $f$ in the central variable defined as in (\ref{centralft}).
The Plancherel formula now takes the following form
\bes 
\int_\Lambda{\|\pi_{\nu}(f)\|^2_{HS}|Pf(\nu)|d\nu} = \int_G{|f(v,z)|^2~dv~dz},
\ees 
which holds for all $L^2$-functions on $G$ by density argument.

\subsection{Uncertainty Principles of Ingham and Paley-Wiener on Step Two Nilpotent Lie Groups}

In order to state and prove analogues of the results of Ingham and Paley-Wiener on connected, simply connected two step nilpotent Lie groups, first we need to prove two lemmas which are essentially modified versions of the Euclidean results with a polynomial occurring in the estimate of the Fourier transform. It is a well known fact that the Pfaffian, which occurs in the Plancherel formula, is a homogeneous polynomial in its variables. This is the reason behind the polynomials occurring in the estimate of the Fourier transform unlike the Euclidean case.  We shall often consider a radial function on $\R^n$ as a function on $[0, \infty)$. 

\begin{lem}\label{openpoly}
Let $P$ be a polynomial on $\R^n$ and $\theta : \R^n \ra [0, \infty)$ be a decreasing radial function with $\lim_{|y| \ra \infty} \theta(y) = 0$ and 
\bes
I=\int_{|y|\geq 1}\frac{\theta (y)}{|y|^n}dy.
\ees
\begin{enumerate}
\item[(a)] Let $f\in L^1(\R^n)$ be a nontrivial function satisfying the estimate
\be \label{estlem}
|\widehat{f}(y)|\leq C |P(y)| e^{-|y|\theta(y)}, \:\:\:\: \txt{ for all } y \in \R^n.
\ee
If $f$ vanishes on a nonempty open set then $I$ must be finite. 
\item[(b)] If $I$ is finite then there exists a nontrivial $f\in C_c^{\infty}(\R^n)$ satisfying (\ref{estlem}).
\end{enumerate}
\end{lem}

\begin{proof}
We write the polynomial 
\bes
P(y) = \sum_{|\alpha| \leq N} a_{\alpha} y^{\alpha}, \:\:\:\: \txt{ for } a_{\alpha} \in \C, \: N \in \N,
\ees 
where 
\beas 
y^{\alpha} &=& y_1^{\alpha_1} \cdots y_n^{\alpha_n}, \:\:\:\: \txt{ for } y = (y_1, \cdots, y_n) \in \R^n, \\
\alpha &=& (\alpha_1, \cdots, \alpha_n) \in \{ \N \cup \{0\} \}^n.
\eeas
Now we define the differential operator 
\bes 
D_P =  \sum_{|\alpha| \leq N} \frac{a_{\alpha}}{(2\pi)^{|\alpha|}} \frac{\partial^{\alpha_1}}{\partial x_1^{\alpha_1}} \cdots  \frac{\partial^{\alpha_n}}{\partial x_1^{\alpha_n}}.
\ees 
It follows that 
\be \label{polyft} 
|\widehat{(D_P \phi)}(y)| = C |P(y)| |\widehat{\phi}(y)|, \:\:\:\: \txt{ for } \phi \in C_c^{\infty}(\R^n), \: y \in \R^n.
\ee
Let $f$ satisfy the hypothesis in (a). We consider $\phi \in {C_c}^\infty(\R^n)$ such that $f*\phi$ is nontrivial and vanishes on an open set in $\R^n$. For $y$ not in the zero set of $P$ from (\ref{polyft}) and (\ref{estlem}) we have 
\be \label{convest}
|\widehat{(f*\phi)}(y)| = C \frac{1}{|P(y)|}|\widehat{(D_P \phi)}(y)||\widehat{f}(y)| \leq Ce^{-|y|\theta(y)}.
\ee
This is true for almost every $y \in \R^n$. Since $\widehat{f}(y)$ is zero for $y$ in the zero set of $P$ the estimate (\ref{convest}) is true for all $y \in \R^n$. Applying Theorem \ref{open} to the function $f * \phi$ we get that $I$ is finite. To prove (b) let us assume that $I$ is finite. By Theorem \ref{open}, we can find a nonzero $f_1 \in C_c^{\infty}(\R^n)$ such that 
\bes 
|\widehat{f_1}(y)|\leq Ce^{-|y|\theta(y)}, \:\:\:\: \txt{ for } y \in \R^n. 
\ees
From (\ref{polyft}) we get that $f = D_P f_1 \in C_c^{\infty}(\R^n)$ satisfies the estimate
\bes 
|\widehat{f}(y)|\leq C |P(y)| e^{-|y|\theta(y)}, \:\:\:\: \txt{ for } y \in \R^n.
\ees 
\end{proof}
\begin{lem}\label{halfsppoly}
Let $P$ be a polynomial on $\R^n$ and $\psi: \R^n \ra [0, \infty)$ be a locally integrable radial function with  
\bes 
I = \int_{\R^n}\frac{\psi(x)}{(1+|x|)^{n+1}}dx.
\ees
\begin{enumerate}
\item[(a)] Let $f\in L^2(\R^n)$ be a nontrivial function satisfying the estimate
\be \label{lemest} 
|\widehat f(y)|\leq C|P(y)| e^{-\psi(y)} , \:\:\:\: \txt{ for almost every } y\in \R^n.
\ee
If $f$ is supported on a half space then $I$ must be finite. 
\item[(b)] If $\psi$ is nondecreasing and $I$ is finite then there exists a nontrivial $f\in C_c^{\infty}(\R^n)$ satisfying (\ref{lemest}).
\end{enumerate}
\end{lem}

\begin{proof}
It can be proved in a manner similar to the proof of Lemma \ref{openpoly}.
\end{proof}

Let $G$ be a connected, simply connected two step nilpotent Lie group. We can identify $\Lambda$ with $\R^{r+k}$ as measure spaces by identifying $\widetilde{A}^*_Q$ with $\R^r$ and $\mathcal U$ with a full measure set in $\R^k$. We now present analogues of Theorem \ref{open} and Theorem \ref{halfspace} respectively for connected, simply connected two step nilpotent Lie groups.

\begin{thm}\label{openG}
Let $\theta: \R^k \ra [0,\infty)$ be a radial decreasing function with $\lim_{|t| \ra \infty} \theta(t) = 0$ and 
\bes 
I=\int_{|t|\geq 1}\frac{\theta(t)}{|t|^k}dt. 
\ees
\begin{enumerate}
\item[(a)] Let $f\in{L^1(G)\cap L^2(G)}$ be such that 
\be \label{estimate} 
\|\pi_{\mu,\nu}(f)\|_{HS}\leq C |H(\mu)| |Pf(\nu)|^{1/2} e^{-|\nu|\theta(\nu)}, \:\:\:\: \txt{ for all } (\mu, \nu) \in \Lambda,
\ee 
where $H \in L^1(\R^r) \cap L^2 (\R^r)$. Suppose there exists $\epsilon> 0$ such that $f(v,z)$ is zero for $|z| < \epsilon$ and all $v \in \mathfrak v$. If $I=\infty$ then $f=0$.
\item[(b)] Suppose $G$ is a MW group. If $I$ is finite then there exists a nontrivial $f\in C_c(G)$ satisfying the estimate 
\bes 
\|\pi_{\nu}(f)\|_{HS}\leq C |Pf(\nu)|^{1/2} e^{-|\nu|\theta(\nu)}, \:\:\:\: \txt{ for } \nu \in \mathcal U. 
\ees
\end{enumerate}
\end{thm}
\begin{proof}
For a fixed $\phi \in C_c(\mathfrak v)$ we consider the function $F_\phi$ on $\mathfrak z$ defined by 
\bes 
F_\phi(z)=\int_{\mathfrak v}{f(v,z)\overline{\phi(v)}~dv} , \:\:\:\: \txt{ for almost every } z \in \mathfrak z. 
\ees 
Clearly $F_\phi \in L^1(\R^k)$.
For $\nu \in \mathcal U$ the Euclidean Fourier transform $\widehat{F_{\phi}}$ of $F_{\phi}$ is given by 
\bes
\widehat F_\phi(\nu)=\int_{\mathfrak v}{f^{\nu}(v)\overline{\phi(v)}~dv},
\ees
where $f_{\nu}$ is defined as in (\ref{centralft}). Applying Cauchy-Schwartz inequality and using (\ref{partplancherel}) and (\ref{estimate}) we get that for $\nu \in \mathcal U$
\bes
|\widehat F_\phi(\nu)|^2 \leq C \int_{\mathfrak v}{|f^{\nu}(v)|^2 dv} \leq C|Pf(\nu)|^2 \|H\|_{L^2(\R^r)}^2 e^{-2|\nu|\theta(\nu)}. 
\ees
It follows that 
\bes 
|\widehat F_\phi(\nu)| \leq C |Pf(\nu)| e^{-|\nu|\theta(\nu)}, \:\:\:\: \txt{ for } \nu \in \mathcal U.
\ees
Clearly $F_\phi(z)$ is zero if $|z|<\epsilon$. Applying Lemma \ref{openpoly} to the function $F_{\phi}$ on $\R^k$ we get that $F_\phi$ is zero. Since this is true for all $\phi \in C_c(\mathfrak v)$ we conclude that $f$ is zero. This proves (a). 

We shall now prove (b). By Lemma \ref{openpoly} there exists $g\in C_c^\infty(\R^k)$ satisfying the estimate
\bes 
|\widehat{g}(\nu)|\leq C|Pf(\nu)| e^{-\theta(\nu)|\nu|}, \:\:\:\: \txt{ for } \nu \in \R^k, 
\ees 
where $Pf(\nu)$ is the Pfaffian which is a polynomial in $\nu$.
We consider a fixed $h\in C_c(\mathfrak v)$ and define 
\bes 
f(v,z)= h(v)g(z), \:\:\:\: \txt{ for } v\in \mathfrak v, \: z\in \mathfrak z.
\ees 
Clearly $f\in C_c(G)$ and for $\nu\in \Lambda$ $\pi_{\nu}(f)$ is a Hilbert-Schmidt operator. By (\ref{hs}) we get that 
\bes
 |Pf(\nu)|\|\pi_{\nu}(f)\|^2_{HS} = \int_{\mathfrak v}{|f^\nu(v)|^2 ~dv} = |\widehat g(\nu)|^2 \int_{\mathfrak v}{|h(v)|^2 dv} \leq C |Pf(\nu)|^2 e^{-2 |\nu|\theta (\nu)}.
\ees
Hence we obtain 
\bes 
\|\pi_{\nu}(f)\|_{HS} \leq C |Pf(\nu)|^{1/2} e^{- |\nu|\theta (\nu)}, \:\:\:\: \txt{ for } \nu \in \mathcal U.
\ees
\end{proof}

\begin{thm}\label{halfspG}
Let $\psi : \R^k \ra [0, \infty)$ be a locally integrable radial function and 
\bes 
I = \int_{\R^k}{\frac{\psi(t)}{(1+|t|)^{k+1}}dt}.
\ees 
\begin{enumerate}
\item[(a)] Suppose $f\in L^2(G)$ satisfies the estimate
\bes 
\|\pi_{\mu,\nu}(f)\|_{HS}\leq C |H(\mu)| |Pf(\nu)|^{1/2} e^{-\psi(\nu)}, \:\:\:\: \txt{ for almost every } (\mu, \nu) \in \Lambda,
\ees 
where $H \in L^1(\R^r) \cap L^2 (\R^r)$. 
Suppose that $\txt{supp}~f \subset \{(v,z) \in G\mid z\cdot\eta \leq p\}$ for some $\eta\in S^{k-1}$ and some $p\in \R$. If $I=\infty$ then $f(g) = 0$ for almost every $g\in G$. 
\item[(b)] Suppose $G$ is a MW group. If $I$ is finite and $\psi$ is nondecreasing then there exists a nontrivial $f \in C_c(G)$ satisfying the estimate
\bes
\|\pi_{\nu}(f)\|_{HS}\leq C |Pf(\nu)|^{1/2} e^{-\psi(\nu)}, \:\:\:\: \txt{ for } \nu \in \mathcal U.
\ees
\end{enumerate}
\end{thm}

\begin{proof}
For the proof of (a), we again consider the function $F_{\phi}$ on $\mathfrak z$ as in the proof of Theorem \ref{openG} (a) and get the estimate
\bes 
|\widehat F_\phi(\nu)| \leq C |Pf(\nu)|^{1/2} e^{-\psi(\nu)}, \:\:\:\: \txt{ for } \nu \in \mathcal U. 
\ees 
Applying Lemma \ref{halfsppoly} to $F_\phi$ we similarly obtain that $f=0$. Proof of (b) is similar to that of Theorem \ref{openG}(b).
\end{proof}

\section{Unique Continuation Property of Solutions to the Schr\"odinger Equation}
 
We consider the initial value problem for the time-dependent Schr\"odinger equation on $\R^n$ given by
\bea \label{LA}
\left\{\begin{array}{rcll} 
\displaystyle{\frac{\partial w}{\partial t}(x,t) ~ - ~ i\Delta w (x,t) } &=& ~ 0,
&\txt{ for } (x,t) \in \R^n \times \R, \\
w(x,0) &=& f(x), &\txt{ for } x \in \R^n.
\end{array}\right.
\eea 
Our aim is to obtain sufficient conditions on the behaviour of the solution $u$ at
two different times $t = 0$ and $t = t_0$ which guarantee that $u \equiv 0$ is the unique solution of the above equation. It has recently been observed that uncertainty principles can be used to obtain such sufficient conditions. We refer the reader to \cite{EKPV} and the references therein for results in this regard. These results were further generalized in the context of noncommutative groups in  \cite{BTD, C, LM, PS}. In this section we wish to relate the theorems of Ingham and Paley-Wiener to the above mentioned problem on a connected, simply connected two step MW group. 

Let $A = (a_{ij})$ be a real, symmetric matrix of order $n$ and we define the differential operator $\Delta_A$ on $\R^n$ as 
\bes
\Delta_A = \sum_{i,j=1}^n a_{ij}\frac{\partial^2}{\partial x_i \partial x_j}. 
\ees
Consider the Schr\"odinger equation corresponding to $\Delta_A$ given by
\bea \label{LA}
\left\{\begin{array}{rcll} 
\displaystyle{\frac{\partial w}{\partial t}(x,t) ~ - ~ i\Delta_A w (x,t) } &=& ~ 0,
&\txt{ for } (x,t) \in \R^n \times \R, \\
w(x,0) &=& f(x), &\txt{ for } x \in \R^n.
\end{array}\right.
\eea 
We first prove a unique continuation result for the solution of the Schr\"odinger equation (\ref{LA}) using Theorem {\ref{halfspace}} in the context of $\R^n$.

\begin{thm}\label{schrRn}
Let $w$ be a solution to the equation (\ref{LA}) and $\psi:\R^n \ra [0,\infty)$ be a non-negative, locally integrable, radial function. Assume that $f\in C_c^\infty(\R^n)$ and 
\be \label{LAdecay}
|w(x,t_0)| \leq Ce^{-\psi(x)}, \:\:\:\: \textmd{ for some } t_0 \neq 0, ~ x\in\R^n.
\ee  
If
\be\label{psiintegral}
\int_{0}^{\infty}{\frac{\psi(r)}{1+r^2}dr}=\infty,
\ee 
then $w=0$. 
\end{thm}

To prove the theorem we shall need the following lemma corresponding to the case when $A$ is a diagonal matrix of the form 
\be \label{diagonal}
A= 
\left(\begin{matrix} 
a_1 & 0 & \cdots & 0 \\ 
0 & a_2 & \cdots & 0 \\ 
\vspace{2mm}\\ 
0 & \cdots & \cdots & a_n 
\end{matrix}\right),
\ee
where each $a_1, \cdots, a_k$ is nonzero and each $a_{k+1}, \cdots, a_n$ is zero for some $1 \leq k \leq n$. We note that $A$ is nonsingular if $k$ equals $n$ and $A$ is singular if $k$ is strictly less than $n$. Let $\sigma$ be the difference of the number of positive eigenvalues of $A$ and the number of negative eigenvalues of $A$. 

\begin{lem}\label{lemma}
If $A$ is of the form (\ref{diagonal}) the fundamental solution of the equation (\ref{LA}) is given by  
\be \label{kernel}
w(x,t) = e^{it\Delta_A}f(x) = f * \beta_t(x),  \:\:\:\: \txt{ for } x \in \R^n.
\ee
\begin{enumerate}
\item[(a)] If $A$ is nonsingular the kernel $\beta_t$ is given by the function $\gamma_t$ on $\R^n$ defined as 
\be \label{gamma}
\gamma_t(x)=\frac{1}{\sqrt{|a_1\cdots a_n|}{(4\pi|t|)}^{\frac{n}{2}}}e^{-i\pi\frac{\sigma}{4}}e^{\frac{i}{4t}\sum_{j=1}^n\frac{x_j^2}{a_j}}, \:\:\:\: \txt{ for } x = (x_1, \cdots, x_n) \in \R^n.
\ee
\item[(b)] If $A$ is singular the kernel $\beta_t$ is given by a tempered distribution on $\R^n$ defined as
\bes
\beta_t(\phi) = \int_{\R^k} \phi(x',0) \gamma_t(x') dx', \:\:\:\: \txt{ for } \phi \in \mathcal S(\R^n), ~ x' \in \R^k \txt{ and } 0 \in \R^{n-k}.
\ees 
\end{enumerate}
\end{lem}

\begin{proof}
In this proof we shall use the same notation to denote the Euclidean Fourier transform on $\R^n$ or $\R^k$ which will be evident from the context. The proof of (a) follows easily because the Fourier transform of $\gamma_t$ is given by (see Ch. 4, \cite{W}) 
\bes
\widehat{\gamma_t}(\xi) = e^{-4it\pi^2\sum_{j=1}^n a_j\xi_j^2}, \:\:\:\: \txt{ for } \xi=(\xi_1, \cdots, \xi_n) \in \R^n.
\ees
It remains to prove (b). For a given test function $\phi \in \mathcal{S}(\R^n)$ we define a function $h_{\phi}$ on $\R^k$ by
\bes
h_{\phi}(x') = \int_{\R^{n-k}} \phi(x',x'') ~ dx'',  \:\:\:\: \txt{ for } x' \in \R^k.
\ees 
It is easy to see that for $\xi' \in \R^k$ and $0 \in \R^{n-k}$
\bes
\widehat{\phi}(\xi',0) =  \int_{\R^k} \left( \int_{\R^{n-k}} \phi(x',x'') ~ dx'' \right) e^{-2\pi i x' \cdot \xi'} dx'
= \widehat{h_{\phi}}(\xi').
\ees
For a given test function $\phi \in \mathcal{S}(\R^n)$ we get that 
\beas
\widehat{\beta_t}(\phi) 
&=& \beta_t(\widehat{\phi}) \\
&=& \int_{\R^k} \widehat{\phi}(\xi',0) ~ \gamma_t(\xi') ~ d\xi' \\
&=&  \int_{\R^k} \widehat{h_{\phi}}(\xi')~ \gamma_t(\xi') ~ d\xi' \\
&=& \int_{\R^k} h_{\phi}(\xi') ~ \widehat{\gamma_t}(\xi') ~ d\xi' \\
&=&  \int_{\R^k} \int_{\R^{n-k}} \phi(\xi',\xi'') ~ \widehat{\gamma_t}(\xi') ~ d\xi'' ~ d\xi'.
\eeas
It follows that the distributional Fourier transform of $\beta_t$ is given by the function
\bes
\widehat \beta_t(\xi)= \widehat {\gamma_t}(\xi'), \:\:\:\: \txt{ for } \xi = (\xi',\xi''), ~ \xi' \in \R^k, ~ \xi'' \in\R^{n-k}. 
\ees 
On the other hand it is easy to see that the convolution of $f \in C_c^{\infty}(\R^n)$ with the tempered distribution $\beta_t$ is given by the function
\be \label{conv}
f * \beta_t(x) = \int_{\R^k} f(y', x'') \gamma_t(x'-y') dy', \:\:\:\: \txt{ for } x = (x',x''), ~ x' \in \R^k, ~ x'' \in \R^{n-k}.
\ee
It turns out that the tempered distribution $\beta_t$ satisfies (\ref{kernel}) because of the following equality
\bes
\widehat{(e^{it{\Delta_A}}f)}(\xi) = e^{-4\pi^2it\sum_{j=1}^k a_j\xi_j^2}\widehat f(\xi) = \widehat{\gamma_t}(\xi') \widehat f(\xi) =\widehat \beta_t(\xi) \widehat f(\xi), 
\ees
where $\xi = (\xi',\xi''), ~ \xi' = (\xi_1, \cdots, \xi_k) \in \R^k, ~ \xi'' \in\R^{n-k}. $
\end{proof}

\vspace{0.1in}

\noindent \textit{Proof of Theorem \ref{schrRn}}. We shall first prove the theorem for the special case when $A$ is of the form (\ref{diagonal}). Further, if $A$ is singular it follows from Lemma \ref{lemma} (b) and the relation (\ref{conv}) that 
\be\label{kernelgamma}
w(x,t_0) = e^{it_0\Delta_A}f(x) = f_{x''}*\gamma_{t_0}(x'), \:\:\:\: \txt{ for } x = (x',x''), ~ x' \in \R^k, ~ x'' \in \R^{n-k},
\ee
where for fixed $x'' \in\R^{n-k}$, $f_{x''}$ is defined on $\R^k$ as 
\bes 
f_{x''}(x')=f(x',x''), \:\:\:\: \txt{ for } x' \in \R^k.
\ees
Note that the convolution above is on $\R^k$. Using the expression of $\gamma_t$ given in (\ref{gamma}) we obtain
\bes
f_{x''}*\gamma_{t_0}(x') = \frac{e^{-i\pi \sigma/4} e^{\frac{i}{4t_0}\sum_{j=1}^{k} x_j^2/a_j}} {\sqrt{|a_1 \cdots a_k|}(4\pi|t_0|)^{\frac{k}{2}}} \int_{\R^k}f_{x''}(y')e^{\frac{i}{4t_0}\sum_{j=1}^{k} \frac{y_j^2}{a_j}}e^{-2\pi i\sum_{j=1}^{k}x_j\frac{y_j}{4\pi t_0a_j}}dy'.
\ees
It is easy to see using change of variables $z_j = y_j/(4\pi t_0 a_j) $ for each $1 \leq j \leq k$ that
\be \label{hatgx''}
|f_{x''}*\gamma_{t_0}(x')| = \sqrt{|a_1\cdots a_k|}(4\pi|t_0|)^{\frac{k}{2}}|\widehat{g_{x''}}(x')|,
\ee
where $g_{x''}$ is defined on $\R^k$ as 
\bes
g_{x''}(z')=f_{x''}(4\pi t_0 a_1z_1, \cdots, 4 \pi t_0 a_k z_k)e^{4\pi^2 it_0\left(a_1 z_1^2 +\cdots + a_k z_k^2\right)}, \:\:\:\: \txt{ for } z'=(z_1, \cdots, z_k) \in \R^k.
\ees
Since $f \in C_c^{\infty}(\R^n)$, it follows that $f_{x''} \in C_c^{\infty}(\R^k)$ for each $x'' \in \R^{n-k}$ and therefore so is $g_{x''}$.
From (\ref{hatgx''}), (\ref{kernelgamma}) and (\ref{LAdecay}) we get that for fixed $x'' \in \R^{n-k}$
\bes
|\widehat {g_{x''}}(x')| \leq  C e^{-\psi_{x''}(x')}, \:\:\:\: \txt{ for } x' \in \R^k,
\ees
where $\psi_{x''}(x')= \psi(x',x'')$ for $x'\in \R^k$. Since $\psi$ is a radial function on $\R^n$, using the change of variable $s^2 + |x''|^2 = r^2$ it follows from (\ref{psiintegral}) that
\beas
\int_{|x'| \geq M} \frac{\psi_{x''}(x')}{|x'|^{k+1}} dx' &=& C \int_M^{\infty} \frac{\psi(\sqrt{s^2 + |x''|^2})}{s^{k+1}} s^{k-1}ds \\
&=& C \int_{\sqrt{M^2+|x''|^2}}^\infty \frac{\psi(r)}{r^2 - |x''|^2}\frac{r}{\sqrt{r^2-|x''|^2}}dr\\
&\geq& C \int_{\sqrt{M^2+|x''|^2}}^\infty \frac{\psi(r)}{1 + r^2}dr\\
&=& \infty.
\eeas
Applying Theorem \ref{halfspace} to the function $g_{x''}$ on $\R^k$ we get that $g_{x''}$ is zero on $\R^k$ and hence so is $f_{x''}$. Since this is true for each $x'' \in \R^{n-k}$ it follows that $f$ is zero on $\R^n$ and hence so is $w$. If $A$ is nonsingular the result follows in a similar fashion from Lemma \ref{lemma} (a). This proves the theorem for the special case when $A$ is of the form (\ref{diagonal}).

Now we consider $A$ to be a real symmetric matrix of order $n$, not necessarily diagonal. However there exists an orthogonal matrix $P=(P_{ij})$ of order $n$ such that $D = P^tAP$ is a diagonal matrix of the form (\ref{diagonal}) having diagonal entries $d_1, \cdots, d_n$ where each $d_1, \cdots, d_k$ is nonzero and each $d_{k+1}, \cdots, d_n$ is zero for some $1 \leq k \leq n$. Here $P^t$ denotes the transpose of the orthogonal matrix $P$ and by orthogonality $P^t$ equals $P^{-1}$. Now we shall consider the operator $\Delta_A$ in terms of the variable $v= P^tx$. Indeed, if $X = \left(\frac{\partial}{\partial x_1}, \cdots, \frac{\partial}{\partial x_n}\right)^t,$ $V= \left(\frac{\partial}{\partial v_1}, \cdots, \frac{\partial}{\partial v_n}\right)^t$ then $P^t X = V $ and  
\be \label{DeltaDv}
\Delta_A = \sum_{i,j=1}^n a_{ij}\frac{\partial^2}{\partial x_i \partial x_j} = X^t A X = (P^t X)^t D (P^t X) = V^t D V = \sum_{j=1}^k d_j \frac{\partial^2}{\partial v_j^2} = \Delta_{D,v},
\ee 
where the operator $\Delta_{D,v}$ is of the diagonal form corresponding to the diagonal matrix $D$ with respect to the variable $v$. Let us define a function $f_P \in C_c^{\infty}(\R^n)$ by 
\bes
f_P(x)=f(Px), \:\:\:\: \txt{ for } x \in \R^n.
\ees
Using chain rule, the relation $V =P^t X$ and the orthogonality of the matrix $P$ it is easy to see that for $1 \leq j \leq n$ and $x \in \R^n$
\be \label{chainrule}
\frac{\partial f_P}{\partial x_j}(P^tx)
= \left(P_{1j} \frac{\partial f}{\partial x_1}+\cdots + P_{nj} \frac{\partial f}{\partial x_n}\right)(PP^tx)
= \frac{\partial f}{\partial v_j}(x).
\ee
It follows from (\ref{DeltaDv}) and (\ref{chainrule}) that for $x \in \R^n$
\be \label{DeltaAD}
\sum_{i,j=1}^n a_{ij}\frac{\partial^2 f}{\partial x_i \partial x_j}(x) = \sum_{j=1}^k d_j \frac{\partial^2 f}{\partial v_j^2}(x) = \sum_{j=1}^k d_j \frac{\partial^2 f_P}{\partial x_j^2}(P^tx) = \Delta_D f_P(P^tx),
\ee
where the operator 
\bes
\Delta_D = \sum_{j=1}^k d_j \frac{\partial^2}{\partial x_j^2}
\ees 
corresponds to the diagonal matrix $D$ with respect to the variable $x$. Taking Fourier transform in (\ref{DeltaAD}) we get that
\bes
\sum_{i,j=1}^n a_{ij}\xi_i\xi_j \widehat{f}(\xi) =  \sum_{j=1}^k d_{j}\eta_j^2 \widehat{f_P}(\eta),
\ees
where $\eta = (\eta_1, \cdots, \eta_n) \in \R^n$ and $\xi = (\xi_1, \cdots, \xi_n) \in \R^n$ are related by $\eta=P^t\xi$. Consequently, it is easy to see that
\bes
e^{it_0\sum_{i,j=1}^n a_{ij}\xi_i\xi_j} \widehat{f}(\xi) = e^{it_0\sum_{j=1}^k d_{j}\eta_j^2} \widehat{f_P}(\eta),  \:\:\:\: \txt{ for } \xi = P \eta \in \R^n.
\ees
Hence it follows that  
\be \label{expDeltaAD}
e^{it_0\Delta_A }f(x) = e^{it_0\Delta_D}f_P(P^tx), \:\:\:\: \txt{ for } x \in \R^n.
\ee
Using (\ref{LAdecay}), (\ref{kernel}), (\ref{expDeltaAD}) and the radiality of $\psi$ we get that
\bes
|e^{it_0\Delta_D }f_P(x)|\leq Ce^{-\psi(x)}, \:\:\:\: \txt{ for } x \in \R^n.
\ees
By what we have already proved in the previous case for diagonal matrices of the form (\ref{diagonal}), we can conclude that the function $f_P \in C_c^{\infty}(\R^n)$ is zero. It follows that $f$ is zero and hence so is $w$.

\begin{rem}
If we consider $A$ to be nonsingular and diagonal, stronger unique continuation results corresponding to Theorem \ref{open} and Theorem \ref{halfspace} can be proved. Precisely, for the result corresponding to Theorem \ref{open}, if we assume that $\psi(x) = |x| \theta(x)$ for $x \in \R^n$ and $I = \infty$ where $\theta$, $I$ are as in Theorem \ref{open} and if $f \in \mathcal S(\R^n)$ vanishes on any open set in $\R^n$, then $w$ is zero. On the other hand, for the result corresponding to Theorem \ref{halfspace}, if $f \in \mathcal S(\R^n)$ vanishes on any half space in $\R^n$, then $w$ is zero.
\end{rem}

We shall now prove a unique continuation result for solution of Schr\"odinger equation in the context of connected, simply connected two step nilpotent Lie group with MW condition using Theorem \ref{open} and Theorem \ref{halfspace}. Our method of proof is  similar to that in \cite{LM} and will use some of their notation and calculations. For more details see \cite{LM} and the references therein. 

Let $G$ be a connected, simply connected two step nilpotent Lie group with MW condition. Corresponding to any $X$ in the Lie algebra $\mathfrak g$ of $G$, there exists a left invariant differential operator acting on $f \in C^{\infty}(G)$ defined as
\bes 
Xf(g) = \left.\frac{d}{dt}\right|_{t=0} f(g\exp{tX}), \:\:\:\: \txt{ for } g \in G. 
\ees
Let $\{X_i, ~ i = 1, \cdots , m\}$ be an orthonormal basis of $\mathfrak v$ as defined in Section 3.1 and define the left invariant differential operator 
\bes
\mathcal L =\sum_{i=1}^m X_i^2
\ees 
to be the sublaplacian on $G$. The Schr\"odinger equation in this case is given by
\bea \label{schrG}
\left\{\begin{array}{rcll} 
\ds{\frac{\partial w}{\partial t}(g,t)} &=& i\mathcal L w(g,t),  \:\:\:\: &\txt{ for } g\in G, ~ t \in \R, \\
w(g,0) &=& f(g), \:\:\:\: &\txt{ for } g \in G,
\end{array}\right.
\eea
where $f\in L^2(G)$. If we assume that $w \in C^1(\R, L^2(G))$ then the solution to the Schr\"odinger equation (\ref{schrG}) is given by 
\bes
w(g,t) = e^{it\mathcal L}f(g), \:\:\:\: \txt{ for } g \in G, \:\: t \in \R.
\ees
Our main aim in this section is to prove a unique continuation result for solutions to the  Schr\"odinger equation (\ref{schrG}). We shall need few results from \cite{LM} and the notation introduced in Section 3. For a MW group $G$, we have seen that $\Lambda = \{ \nu \in \mathfrak z^* \mid B_{\nu} \txt{ is nondegenerate} \}$. It is known that on an open, dense subset $U$ of the unit sphere in $\frak{z}^*$, the almost symplectic basis $\{ X_1(\nu), \cdots, X_m(\nu)\}$ can be chosen to depend analytically on $\nu$, that is, the matrix coefficients of the transformation matrix from the basis $\{ X_1, \cdots, X_m\}$ to $\{ X_1(\nu), \cdots, X_m(\nu)\}$ are analytic functions of $\nu$ (see P. 2107, \cite{LM}). We define an open set $\mathcal C = \Lambda \cap \R^*_+ U$ where $ \R^*_+ U = \{ \lambda \nu \mid \lambda >0, \nu \in U\}$. For $\nu \in \mathcal C$, $t \in \R$ and $v \in \mathfrak v$ we define 
\be \label{f_t^nu}
f^\nu_t(v)=e^{i\frac{\pi}{2} \left\langle v, S_{\nu} coth\left(tS_{\nu}/2\right)v \right\rangle} f^\nu(v).
\ee 
The following expression (see P. 2111, \cite{LM}) relates the Fourier transform in the central variable of the solution to (\ref{schrG}) at the time $t/4\pi$ with the Euclidean Fourier transform $\widehat{f_t^{\nu}}$ of $f_t^{\nu}$ as follows
\be \label{rel}
\left(e^{i\frac{t}{4\pi}\mathcal L}f\right)^{\nu}(v) = c_{t,\nu} ~ e^{i\frac{\pi}{2}\left\langle v, S_{\nu} coth\left(tS_{\nu}/2\right)v\right\rangle}  \widehat f^\nu_t(v'), 
\ee 
where 
\be \label{v'}
v' = \frac{S_{\nu}}{2} (\coth(tS_{\nu}/2)+I)v,
\ee
and $c_{t,\nu}$ is a constant depending on $t,\nu$. 

\begin{thm}\label{uniqG}
Let w be the solution to the system (\ref{schrG}) satisfying
\begin{enumerate}
\item[(a)] $f(v,z)$ is continuous and compactly supported in the $v$-variable and 
\bes
|f(v,z)| \leq C \frac{e^{-|z|\theta(|z|)}}{(1+|z|^2)^k}, \:\:\:\: \txt{ for all } v \in \mathfrak v, ~ z \in \mathfrak z,
\ees
where $\theta$ is a nonnegative decreasing function on $[0,\infty)$ such that $\lim_{|z| \ra \infty} \theta(z) = 0$.
\item[(b)] 
\bes 
|w((v,z),t_0)|\leq Ce^{-\psi(|v|)} \frac{e^{-|z|\theta(|z|)}}{(1+|z|^2)^k},
\ees 
for some nonzero $t_0\in{\R}$ where $\psi$ is a nonnegative nondecreasing function on $[0,\infty)$. 
\end{enumerate}
If 
\bes 
\ds{\int_{\R}{\frac{\psi(r)}{1+r^2}dr}=\infty} \txt{ and } \ds{\int_{1}^{\infty}{\frac{\theta(r)}{r}dr}=\infty},
\ees 
then $w$ is zero on $G\times \R$.
\end{thm}

\begin{proof}

We shall first show that for $t_1=4\pi t_0\neq0$ and sufficiently small $\nu \in \mathcal C$ the function $f_{t_1}^{\nu}$ defined in (\ref{f_t^nu}) is zero. For small $\nu \in \mathcal C$ we have 
\beas
\coth\left(t_1S_{\nu}/2\right)+I 
&=& (I + O(|\nu|^2))\left(t_1S_{\nu}/2\right)^{-1} + I \\
&=& \left(I + O(|\nu|^2) + t_1S_{\nu}/2\right)\left(t_1S_{\nu}/2\right)^{-1} \\
&=& \left(I + O(|\nu|) \right)\left(t_1S_{\nu}/2\right)^{-1}.
\eeas
It follows that $S_{\nu} (\coth(t_1S_{\nu}/2)+I)$ is invertible and  
\bes 
\left(\frac{S_{\nu}}{2}(\coth(t_1S_{\nu}/2)+I)\right)^{-1} = \left(t_1+ O(|\nu|)\right)I.
\ees
From (\ref{v'}) we obtain for small $\nu \in \mathcal C$
\be \label{vv'}
v = \left(t_1+ O(|\nu|)\right)v'.
\ee
It follows from (\ref{rel}) and (\ref{vv'}) that
\bes 
|\widehat f^\nu_{t_1}(v')|
= |c_{t_1,\nu}|^{-1}\left|\left(e^{it_0\mathcal L} f\right)^\nu(t_1v'+O(|\nu|)v')\right|.
\ees
Since $f$ is continuous and compactly supported in the $v$-variable, by (\ref{f_t^nu}) $f_{t_1}^\nu$ is also compactly supported. Using assumption (b) in the hypothesis and (\ref{centralft}) we have
\be \label{hatf_t^nu1}
|\widehat f^\nu_{t_1}(v')|\leq |c_{t_1,\nu}|^{-1}e^{-\psi(|t_1v'+O(|\nu|)v'|)}.
\ee
Given $\epsilon>0$, we choose $\delta>0$ such that for every $v' \in \mathfrak v$ and $\nu \in \mathcal C$ with $|\nu|<\delta$, we have 
\bes 
\left|t_1v'+O(|\nu|)v'\right| \geq \frac{|t_1||v'|}{1+\epsilon}.
\ees 
Since $\psi$ nondecreasing, we get that 
\bes 
\psi\left(|t_1v'+O(|\nu|)v'|\right) \geq \psi\left(\frac{|t_1||v'|}{1+\epsilon}\right).
\ees
From (\ref{hatf_t^nu1}) we obtain 
\bes 
|\widehat f^\nu_{t_1}(v')|\leq c_{t_1,\nu}e^{-\psi(k|v'|)}, \:\:\:\: \txt{ for some }k>0. 
\ees 
Since $f_{t_1}^\nu$ is compactly supported we can apply Theorem \ref{halfspace} to $f^\nu_{t_1}$ on $\mathfrak v \cong \R^m$ to get that $f^\nu_{t_1}$ is zero. From (\ref{f_t^nu}) it follows that $f^{\nu}$ is zero. This is true for any $\nu$ in $\mathcal C$ which is contained in a sufficiently small neighbourhood of the origin in $\mathfrak z^*$. Now for fixed $v \in \mathfrak v$ consider the function $F_{v}: \mathfrak z^* \ra \C$ given by 
\bes 
F_v(\nu) := f^{\nu}(v), \:\:\:\: \txt{ for } \nu \in \mathfrak z^*.
\ees 
It follows from assumption (a) in the hypothesis that the Euclidean Fourier transform $\widehat{F_v}$ of $F_v$ satisfies the estimate
\bes 
|\widehat{F_v}(z)| = |f(v,-z)| \leq C e^{-|z|\theta(|z|)}, \:\:\:\: \txt{ for } z \in \mathfrak z.
\ees  
Since $f^{\nu}$ vanishes for $\nu$ in a small open set in $\mathcal C$ (thereby open in $ \mathfrak z^*$ since $\mathcal C$ is open), we can apply Theorem \ref{open} to the function $F_v$ on $ \mathfrak z^* \cong \R^k$ to get that $f$ is zero. Hence so is $w$. 

\end{proof}

\begin{rem}
In \cite{LM} Ludwig and M\"uller proved a unique continuation property for solution to the Schr\"odinger equation on any connected, simply connected two step nilpotent Lie group $G$ corresponding to Hardy's uncertainty principle for a large class of second order left-invariant differential operators on $G$ given by
\bes
\mathcal L_A = \sum_{i,j=1}^m a_{ij} V_i V_j,
\ees
where $A=(a_{ij})$ is a real, symmetric matrix of order $m$. They first proved the result assuming $G$ to be a MW group. For a non MW group $G$ they embedded $G$ inside a bigger MW group $H$ and considered a lift $\widetilde{\mathcal L_A}$ on $H$ of the operator $\mathcal L_A$ on $G$. The result then follows from the proof for MW groups in the case where $A$ is singular. It is not very difficult to see that the proof of Theorem \ref{uniqG} can be extended to the case where $A$ is nonsingular for a MW group $G$. However at present it is not clear to us how to extend the proof of Theorem \ref{uniqG} for the case when $A$ is singular.
\end{rem}

\end{document}